\documentclass[sn-mathphys,Numbered]{sn-jnl}

\usepackage{url}
\usepackage{mathrsfs}
\usepackage{float}
\usepackage{graphics}

\usepackage{enumitem}

\usepackage{graphicx}%
\usepackage{multirow}%
\usepackage{amsmath,amssymb,amsfonts}%
\usepackage{amsthm}%
\usepackage{mathrsfs}%
\usepackage[title]{appendix}%
\usepackage{xcolor}%
\usepackage{textcomp}%
\usepackage{manyfoot}%
\usepackage{booktabs}%
\usepackage{algorithm}%
\usepackage{algorithmicx}%
\usepackage{algpseudocode}%
\usepackage{listings}%

\theoremstyle{thmstyleone}%
\newtheorem{theorem}{Theorem}
\newtheorem{proposition}{Proposition}

\theoremstyle{thmstyletwo}%
\newtheorem{remark}{Remark}%

\newtheorem{assumption}{Assumption}

\theoremstyle{thmstylethree}%

\begin{document}

\title[-]{\vspace{-3.5em}\noindent\normalfont\large\itshape Accepted for publication in the Journal of Nonlinear Science.\\[2.5em] \bfseries On a Stochastic PDE Model for Epigenetic Dynamics}



\author[1]{\fnm{Pablo} \sur{Padilla-Longoria}}\email{pablo@mym.iimas.unam.mx}
\equalcont{These authors contributed equally to this work.}

\author*[2]{\fnm{Jesus} \sur{Sierra}}\email{jesus.sierra@cimat.mx}
\equalcont{These authors contributed equally to this work.}

\affil[1]{\orgdiv{IIMAS}, \orgname{Universidad Nacional Autonoma de Mexico
(UNAM)}, \orgaddress{\street{Ciudad Universitaria}, \city{Mexico City}, \postcode{04510}, \country{Mexico}}}

\affil*[2]{\orgname{CIMAT}, \orgaddress{\street{De Jalisco s/n}, \city{Gto.}, \postcode{36023}, \state{Guanajuato}, \country{Mexico}}}

\abstract{
We propose a stochastic model to investigate epigenetic mutations, i.e., modifications of the genetic information that control gene expression patterns in a cell but do not alter the DNA sequence. Epigenetic mutations are related to environmental fluctuations, which leads us to consider (additive) noise as the driving element for such mutations (noise-induced transitions in Waddington's epigenetic landscape). We focus on two applications: firstly, molecular biochemistry of cancer immunology involving macrophages' epigenetic modifications, where we show the relevance of random perturbations in the tumor microenvironment, and secondly, cell fate determination and mutation of the flower Arabidopsis thaliana. Due to the complexities of cancer biology for the first case, we present the details in \cite{sierra2025tumor} since our principal objective here is to validate our system as an appropriate epigenetic model for more general biological applications, with emphasis on mathematical oncology and developmental biology; for such results, we rely on the theory of Stochastic PDE, theory of large deviations, and ergodic theory. Moreover, since epigenetic mutations are reversible, a fact currently exploited to develop so-called epi-drugs to treat diseases such as cancer, we also investigate an optimal control problem for our system to study the reversal of epigenetic mutations; our control problem is also relevant for studying epigenetic stabilizers and transcription factors in immunotherapies for cancer \cite{sierra2025tumor}.
}



\maketitle

\section{Introduction }

In this paper, we propose a stochastic model for epigenetic dynamics. As mentioned above, one of the primary applications of our model is in understanding why the immune system does not `attack' cancer and how to overcome this problem. Due to the distinct characteristics of cancer biology, we present these results in \cite{sierra2025tumor}. The goal of this manuscript is to provide the mathematical validation of the stochastic epigenetic model introduced in \cite{sierra2025tumor}. Our study requires a (global) well-posed solution as a starting point. However, we emphasize that our principal objective is the rigorous validation of the noise-induced dynamics (predicting stochastic phenomena) in such solution leading to epigenetic mutations; since this phenomenon may be complicated to observe (and validate) in a numerical context, we believe that a mathematical approach is fundamental, in particular for the delicate applications that we consider, such as cancer immunology \cite{sierra2025tumor}: our model, along with the stochastic theory presented here, explain why the immune system (particularly macrophages) does not attack cancer but instead helps tumor progression and development. 

We show that our model can be used to study more general biological problems involving epigenetics, especially in oncology and developmental biology. In particular, we model the cell fate determination and (epigenetic) mutation of Arabidopsis thaliana through a stochastic reaction-diffusion system governed by a potential field and additive noise. The potential mimics the flower's epigenetic landscape as defined by Waddington, and the noise represents environmental fluctuations. We show through numerical simulations that the system eventually exits the local minima, traversing the epigenetic landscape in the spatial order that, in many of the realizations, corresponds to the correct architecture of the flower, that is, following the observed geometrical features of the meristem. We use the theory of large deviations to estimate the exit time, characterize the associated invariant measure, and discuss the phenotypic implications. We also investigate an optimal control problem for our system to study the reversal and stabilization of epigenetic mutations. 

There are approximately 250,000 species of flowering plants (Angiosperms).
The organs of the flower in most of them (the only known exception
being the flower Lacandonia schismatica) are organized in four
concentric rings (whorls): sepals, petals, stamens, and carpels (from
the outer rim to the center).

We work with Arabidopsis thaliana, the first plant whose complete
genome was sequenced and has been extensively studied \cite{feldmann2014first}.
In \cite{alvarez2008floral}, using experimental data, the authors
obtained the gene regulatory network (GRN) that determines the fate
of floral organ cells in Arabidopsis thaliana. Based on this
model, Cort{\'e}s-Poza and Padilla-Longoria
constructed a system of reaction\textendash diffusion equations governed
by a potential field corresponding to the epigenetic landscape of
the flower\textquoteright s organ formation \cite{cortes2022variational}. In this work, we want
to introduce perturbations due to epigenetic factors, particularly
environmental fluctuations, into the system as additive noise; our
goal is to study this additional mechanism in connection with mutations
observed in Arabidopsis thaliana. 

We consider an energy landscape with isolated minima. Furthermore,
the deterministic part of the dynamics drives the system by the steepest
descent to the vicinity of one of these minima, where it remains
for a very long time. We will see that the random perturbations push
the system significantly up and away from this minimum. After some
time, the system will manage to escape the basin of attraction of
the minimum it is currently in and find its way toward the location
of another minimum; we identify this transition as an epigenetic mutation.  

We will use the theory of large deviations (see \cite{freidlin1998random})
to estimate the time of escape from the location of one minimum to
another, which is exponentially long relative to the height of the
energy barrier between these minima measured in units of the amplitude
of the random perturbation. We remark that, no matter how small the amplitude of the random perturbation is, we can prove that such exit (noise-induced transition or epigenetic mutation) will occur with probability equal to 1 (Theorem \ref{thm:Thm_large_dev}), with the exit time inversely proportional to the amplitude of the random perturbation: this rigorous result is of significance since environmental perturbations are typically small and we expect the transitions (mutations) to take a very long time to happen, something that would be difficult to observe (and validate) in numerical computations. 

The theory of large deviations also provides information about the paths of maximum likelihood by which the transitions between minima occur (Theorem \ref{thm:Thm_large_dev} (iii)); this is important since the transitions should occur at particular points of the landscape determined by the genetic information of the biological system under consideration. 

It is also relevant to analyze the system's behavior as the noise vanishes, in which case we should expect the system's dynamics to get more and more concentrated around the minima of the landscape; we show this by studying the invariant measure. (Theorem \ref{thm:Thm_large_dev} (ii))

As mentioned before, our results can be applied to study
the connection between environmental fluctuations and diseases such as
Alzheimer's and cancer. It is
now recognized that epigenetics plays a role in the development of
cancer (carcinogenesis) \cite{cancerepig}; see also \cite{esteller2008epigenetics,sharma2010epigenetics}.
For example, abnormal epigenetic modifications in specific oncogenes
and tumor suppressor genes can result in uncontrolled cell growth
and division that can cause cancer. Besides, epigenetic alterations
in regions of DNA outside of genes can also give rise to cancer. It
is now accepted that the environment and human behavior are the principal
causes of abnormal epigenetic modifications. 

Unlike genetic mutations, epimutations are reversible, which gives
the possibility to reverse the epimutations in cancer cells through
the so-called epi-drugs \cite{kristensen2009epigenetics}. The targets
of epigenetic therapy are the enzymes involved in epigenetic modifications.
This motivates the introduction of the theory of control into our
system: our model can help to find and characterize the elements (controls)
needed for such reversal (see, e.g., \cite{lu2021mathematical}).

On the other hand, in \cite{sierra2025tumor} and using the model proposed in this manuscript, we study the epigenetic mechanism (related to the tumor microenvironment, TME)
responsible for increasing tumor-associated macrophages that promote
the occurrence and metastasis of tumor cells, support tumor angiogenesis,
inhibit T cell-mediated anti-tumor immune response, and lead to tumor
progression. 

Macrophages are particularly interesting to study from a stochastic
analysis point of view due to their plasticity in response to environmental
signals. The functional differences of macrophages are closely related
to their plasticity. Moreover, molecules in TMEs are responsible for
regulating macrophages' functional phenotypes.
Such molecular signals are so diverse and random that we consider
it fit to treat them as Gaussian noise that increases in magnitude
as the tumor progresses. Under this assumption, our mathematical model
shows that most tumor-associated macrophages (TAMs) get eventually
polarized into macrophages with phenotypes that favor cancer development
through a process that
we call noise-induced polarization (see Fig. \ref{fig:ep_stab} right). Moreover, following our results related to stochastic optimal control, we propose a mechanism
to promote the appropriate epigenetic stability for immunotherapies
involving macrophages, which includes p53 and APR-246 (eprenetapopt); see Fig. \ref{fig:ep_stab} left. 

\begin{figure}[H]
\begin{centering}
\begin{tabular}{cc}
\includegraphics[scale=0.55]{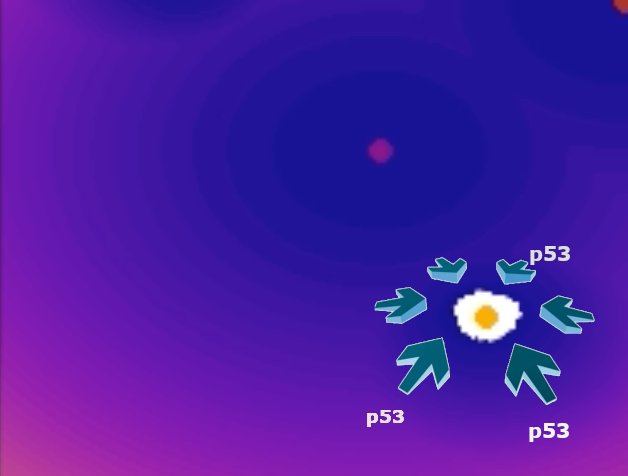} & \includegraphics[scale=0.55]{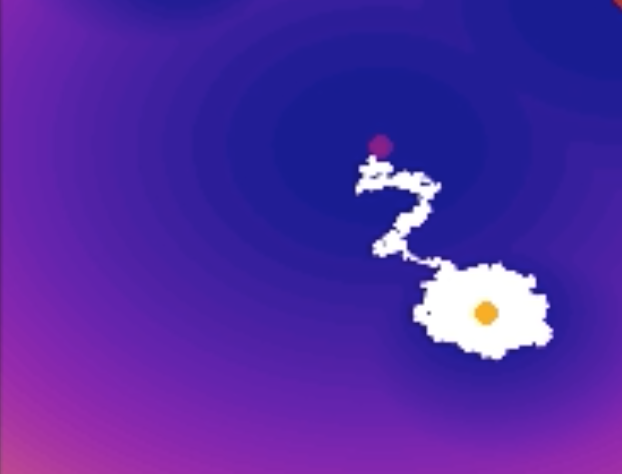}\tabularnewline
\end{tabular}
\par\end{centering}
\caption{The effect of p53 on the epigenetic evolution of macrophages in a
TME. Left: p53 stabilizes the epigenetic M1 (anti-tumoral) state. Right: in the absence
of an epigenetic stabilizer, the macrophage polarizes into a pro-tumoral
state.\label{fig:ep_stab} }

\end{figure}

Finally, we emphasize that the paper's mathematical novelty is not a new large-deviations theorem. Rather, our contribution is the construction and validation of an epigenetic SPDE model for which the abstract Freidlin theory becomes applicable. This requires identifying a biologically meaningful potential landscape, proving that it satisfies the hypotheses of the general theory, interpreting the invariant measure and exit-time asymptotics in terms of epigenetic mutations, and integrating these results with stochastic optimal control.

The rest of the paper is organized as follows. Section \ref{sec:Stoch_model} introduces the stochastic epigenetic model. In Section \ref{sec:Main_results},
we present our main results. We review the necessary theory and notation
in Section \ref{sec:Preliminaries} and demonstrate our results in
Section \ref{sec:Proof_results}. Finally, Section \ref{sec:Numerical_simulations} shows our numerical simulations.

\section{Stochastic Epigenetic Model}\label{sec:Stoch_model}

We model the epigenetic landscape of the flower as a potential field
with four different basins of attraction, each corresponding to a
different flower organ (sepals, petals, stamens, and carpels). For
the size of each basin, we will use the reciprocal of the number of
(dynamically observed) initial conditions that land in each steady
state of the dynamical system (see \cite{cortes2022variational});
this guarantees that equilibrium points that are reached more often
will have larger basins, and conversely, equilibrium points that are
reached fewer times will have a smaller basin. The centers of the
basins are located in $\mathbb{R}^{2}$; see \cite{cortes2022variational}. 

We now define the potential field on the plane $(u,v)$ determined by the epigenetic landscape in the following way:
\begin{equation}
\label{eq:Pot_F}
F(u,v) = \min_{1 \le i \le 4} \left\{ a_{i} \left[ (u-u_{i})^{2} + (v-v_{i})^{2} \right] \right\},
\end{equation}
where $(u_{1},v_{1})$, $(u_{2},v_{2})$, $(u_{3},v_{3})$, and $(u_{4},v_{4})$ are the centers of the basins, and $a_{1}, a_{2}, a_{3}, a_{4}$ define the size of each basin. As in \cite{cortes2022variational}, we consider the smooth (mollified) version of $F$. Moreover, we use a Gaussian filter to smooth out $F$ in our numerical simulations.

\begin{remark}\label{rem_F}
The mollified landscape $F_{\varepsilon}$ satisfies:
\begin{enumerate}[label=\arabic*.]
\item $F_{\varepsilon} \in C^{3}(\mathbb{R}^{2})$,
\item $F_{\varepsilon}(u,v) \rightarrow \infty$ as $|(u,v)| \rightarrow \infty$,
\item The minima are isolated,
\item The Hessian at every minimum is positive definite.
\end{enumerate}
Since the original landscape $F$ is piecewise quadratic with isolated minima, each minimum is realized by a unique quadratic branch in a neighborhood of the minimum. Consequently, for sufficiently small Gaussian mollification parameter $\varepsilon$, the mollified potential $F_{\varepsilon} = G_{\varepsilon} \ast F$ is a $C^{\infty}$ perturbation of the same quadratic polynomial in that neighborhood. Standard properties of mollifiers imply convergence in the $C^{2}$-topology on compact subsets. Therefore, the Hessian at each minimum remains positive definite for sufficiently small $\varepsilon$. Since nondegenerate critical points are structurally stable under small $C^{2}$-perturbations, the minima remain isolated.
\end{remark}

Define $\tilde{f}=-\partial F_{\varepsilon}\left(u,v\right)/\partial u$, $\tilde{g}=-\partial F_{\varepsilon}\left(u,v\right)/\partial v$.
We study the following stochastic reaction-diffusion system with periodic
boundary conditions on $\mathcal{O}=\left[0,R\right],\ R>0$, and
with homogeneous Neumann boundary conditions on $\mathcal{O}=\left(0,R\right)$:
\begin{align}
\frac{\partial u}{\partial t}\left(t,x\right)= & d_{1}\frac{\partial^{2}u}{\partial x^{2}}\left(t,x\right)+\tilde{f}\left(u\left(x,t\right),v\left(x,t\right)\right)+\sigma_{1}w_{1}\left(x,t\right),\nonumber \\
\frac{\partial v}{\partial t}\left(t,x\right)= & d_{2}\frac{\partial^{2}v}{\partial x^{2}}\left(t,x\right)+\tilde{g}\left(u\left(x,t\right),v\left(x,t\right)\right)+\sigma_{2}w_{2}\left(x,t\right),~x\in\mathcal{O},~t>0,\label{eq:sys_epig}
\end{align}
where $d_{1},d_{2}>0$ are diffusion constants and $\sigma_{1}$,
$\sigma_{2}>0$ represent the magnitude of the $noise$. Furthermore,
we consider a white noise perturbation: $w_{k}\left(t,x\right)=\partial^{2}\hat{w}_{k}\left(t,x\right)/\partial t\partial x$,
$k=1,2$, where $\hat{w}_{k}\left(t,x\right)$ are independent Brownian
sheets. We will focus on the case $\sigma_{1}=\sigma_{2}=\sigma>0$
and write $u^{\sigma}\left(t,x\right)=\left(u\left(t,x\right),v\left(t,x\right)\right)$
for the corresponding solution of (\ref{eq:sys_epig}).

We call $M\subset\mathbb{R}^{2}$ a phenotype (or cellular type) region
of the epigenetic landscape if $M$ contains a unique center $\left(u_{k},v_{k}\right)$,
$k\in\left\{ 1,\ldots,4\right\} $, $M$ is closed and convex (hence
simply connected), and $\left(\tilde{f},\tilde{g}\right)$ points
strictly into $M$ on $\partial M$. 

In the deterministic dynamics, each phenotype region coincides with an invariant neighborhood of one minimum of the epigenetic landscape. These regions lie within the corresponding basin of attraction, and every deterministic trajectory that starts inside remains there for all positive time; see Fig. \ref{fig:sigma_zero}.

As in Definition 14.5 \cite{smoller2012shock}, a closed subset, $\Sigma\subset\mathbb{R}^{2}$,
is an invariant region for the deterministic epigenetic system if
any solution $u^{\sigma=0}\left(t,x\right)$ having all of its boundary
and initial values in $\Sigma$, satisfies $u^{\sigma=0}\left(t,x\right)\in\Sigma$
for all $x\in\mathcal{O}$ and $t>0$.

Later, we will show (using the theory of large deviations) that our
epigenetic model can generate a mutation process. Our goal afterward
is to study a mechanism (control) capable of reversing such mutation;
this control's objective will be to return the system from the mutated
state caused by environmental fluctuations. Hence, we look for controls
capable of moving the system state from a mutated (neighborhood of
a) basin to a specific non-mutated one in the epigenetic potential.
This approach can produce significant information for developing therapies
and the so-called epi-drugs to treat diseases such as cancer and Alzheimer's. 

From our previous discussion, we have to consider an optimal control
problem with endpoint/state constraints. Unfortunately, this problem
is not well understood up to now. Hence, we will start with a closely
related optimal control problem for which we can show a Pontriagyn-type
maximum principle. Afterward, we will study the case with endpoint/state
constraints using set-valued analysis and additional restrictions.

For our study of optimal control problems for the stochastic epigenetic
model, it is convenient to adopt the semigroup theory approach of
Da Prato and Zabczyk for stochastic PDE \cite{da2014stochastic}.
We note that, for the Brownian sheets $\hat{w}_{k}\left(t,x\right)$,
$k=1,2$ in (\ref{eq:sys_epig}), the distributional derivative $\partial\hat{w}_{k}\left(t,x\right)/dx$
can be identified, up to a constant, with a cylindrical Wiener process
in $L^{2}\left(\mathbb{T}\right)$; see Section 4.1.5 in \cite{da2014stochastic}.

Let $H$ and $V$ be two separable Hilbert spaces, and denote by $\mathcal{L}_{2}^{0}=\mathcal{L}_{2}^{0}\left(V;H\right)$
the space of Hilbert-Schmidt operators from $V$ into $H$. Let $\left\{ W\left(t\right)\right\} _{t\in\left[0,T\right]}$
be a $V-$valued, $\mathbf{F}-$adapted cylindrical Wiener process
on the filtered probability space $\left(\Omega,\mathcal{F},\mathbf{F},P\right)$,
$\mathbf{F}=\left\{ \mathcal{F}_{t}\right\} _{t\in\left[0,T\right]}$
with standard conditions. Let $\mathbb{F}$ be the progressive $\sigma-$field
(in $\left[0,T\right]\times\Omega$) with respect to $\mathbf{F}$.

Following \cite{lu2021mathematical}, we study the controlled stochastic
PDE
\begin{align}
dx\left(t\right)= & \left(Ax\left(t\right)+a\left(t,x\left(t\right),\alpha\left(t\right)\right)\right)dt\nonumber \\
 & +b\left(t,x\left(t\right),\alpha\left(t\right)\right)dW\left(t\right)\quad\textrm{in }\left(0,T\right],\label{eq:cont_sys}\\
x\left(0\right)= & x_{0},\nonumber 
\end{align}
where $A$ generates a $C_{0}-$semigroup, $\left\{ S\left(t\right)\right\} _{t\geq0}$,
on $H$, $a\left(\cdot\right):\left[0,T\right]\times\Omega\times H\times U\rightarrow H$
and $b\left(\cdot\right):\left[0,T\right]\times\Omega\times H\times U\rightarrow\mathcal{L}_{2}^{0}$;
$U$ is a separable Hilbert space. Moreover, $x\left(\cdot\right)$
is the state variable (valued in $H$) and $\alpha\left(\cdot\right)$
is the control variable (valued in $U$). 

For $\xi\in\mathcal{O}\subset\mathbb{R}$, $d_{1},d_{2}>0$, and appropriate
spaces $H,V,U$ on $\mathcal{O}$ (e.g., $H^{1}\left(\mathcal{O};\mathbb{R}^{2}\right)$,
plus boundary conditions), we will set
\begin{gather}
x\left(t,\xi\right)=\left(\begin{array}{c}
u\left(t,\xi\right)\\
v\left(t,\xi\right)
\end{array}\right),\textrm{ }A=\left(\begin{array}{c}
d_{1}\frac{\partial^{2}}{\partial\xi^{2}}\\
d_{2}\frac{\partial^{2}}{\partial\xi^{2}}
\end{array}\right),\textrm{ }W\left(t\right)=\left(\begin{array}{c}
W_{1}\left(t\right)\\
W_{2}\left(t\right)
\end{array}\right),\textrm{ }\alpha\left(t,\xi\right)=\left(\begin{array}{c}
\alpha_{1}\left(t,\xi\right)\\
\alpha_{2}\left(t,\xi\right)
\end{array}\right),\nonumber \\
e\left(t,x\left(t\right),\alpha\left(t\right)\right)\left(\xi\right)=\left(\begin{array}{c}
e_{1}\left(t,u\left(t,\xi\right),v\left(t,\xi\right),\alpha_{1}\left(t,\xi\right),\alpha_{2}\left(t,\xi\right)\right)\\
e_{2}\left(t,u\left(t,\xi\right),v\left(t,\xi\right),\alpha_{1}\left(t,\xi\right),\alpha_{2}\left(t,\xi\right)\right)
\end{array}\right),\nonumber \\
a\left(t,x\left(t\right),\alpha\left(t\right)\right)\left(\xi\right)=\left(\begin{array}{c}
\tilde{f}\left(u\left(t,\xi\right),v\left(t,\xi\right)\right)\\
\tilde{g}\left(u\left(t,\xi\right),v\left(t,\xi\right)\right)
\end{array}\right)+e\left(t,x\left(t\right),\alpha\left(t\right)\right)\left(\xi\right),\nonumber \\
\left(b\left(t,x\left(t\right),\alpha\left(t\right)\right)w\right)\left(\xi\right)=\sigma\left(\begin{array}{c}
b_{1}\left(t,u\left(t,\xi\right),v\left(t,\xi\right),\alpha_{1}\left(t,\xi\right),\alpha_{2}\left(t,\xi\right)\right)w_{1}\\
b_{2}\left(t,u\left(t,\xi\right),v\left(t,\xi\right),\alpha_{1}\left(t,\xi\right),\alpha_{2}\left(t,\xi\right)\right)w_{2}
\end{array}\right),\nonumber \\
w\left(\cdot\right)=\left(\begin{array}{c}
w_{1}\left(\cdot\right)\\
w_{2}\left(\cdot\right)
\end{array}\right),\label{eq:cont_epig}\\
\textrm{ }\nonumber 
\end{gather} 
where $\tilde{f}=-\partial F_{\varepsilon}\left(u,v\right)/\partial u$, $\tilde{g}=-\partial F_{\varepsilon}\left(u,v\right)/\partial v$, and $W_{1} ,W_{2}$ are independent
cylindrical Wiener processes. Moreover, $\sigma>0$, and $e_{1},e_{2},b_{1,}b_{2}$
have suitable regularity and growth (see below).

We start with the set of controls given by
\[
\mathcal{U}\left[0,T\right]=\left\{ \alpha:\left[0,T\right]\times\Omega\rightarrow U:\alpha\left(\cdot\right)\textrm{ is }\mathbf{F}\textrm{-adapted}\right\} ,
\]
which indicates that our controls are at least nonanticipative; we
will see bellow that, under our assumptions, $\mathcal{U}\left[0,T\right]$
is the set of admissible controls. In \cite{lu2021mathematical},
$U$ can be a separable metric space; however, we only deal with the
Hilbert case. 

Define the cost functional $\mathcal{J}\left(\alpha\left(\cdot\right)\right)$
associated to (\ref{eq:cont_sys}) by 
\begin{equation}
\mathcal{J}\left(\alpha\left(\cdot\right)\right)=E\left[\int_{0}^{T}g\left(t,x\left(t\right),\alpha\left(t\right)\right)dt+h\left(x\left(T\right)\right)\right],\textrm{ }\forall\alpha\left(\cdot\right)\in\mathcal{U}\left[0,T\right].\label{eq:cost}
\end{equation}
Our first objective is to study the following optimal control problem
for the controlled equation (\ref{eq:cont_sys}) with the cost functional
(\ref{eq:cost}):

Find $\bar{\alpha}\left(\cdot\right)\in\mathcal{U}\left[0,T\right]$
such that
\begin{equation}
\mathcal{J}\left(\bar{\alpha}\left(\cdot\right)\right)=\underset{\alpha\left(\cdot\right)\in\mathcal{U}\left[0,T\right]}{\inf}\mathcal{J}\left(\alpha\left(\cdot\right)\right).\label{eq:op_cont_prob}
\end{equation}
We call any $\bar{\alpha}\left(\cdot\right)$ satisfying the last
expression an optimal control; the corresponding state $\bar{x}\left(\cdot\right)$
is an optimal state and $\left(\bar{x}\left(\cdot\right),\bar{\alpha}\left(\cdot\right)\right)$
is an optimal pair.

As mentioned in Section 12.1 \cite{lu2021mathematical}, to establish
Pontryagin-type necessary conditions for an optimal pair $\left(\bar{x}\left(\cdot\right),\bar{\alpha}\left(\cdot\right)\right)$
of Problem (\ref{eq:op_cont_prob}), we have to consider (along with
appropriate assumptions) the following $H-$valued backward stochastic
differential equation:
\begin{align}
dy\left(t\right)= & -A^{*}y\left(t\right)dt-\left(a_{x}\left(t,\bar{x}\left(t\right),\bar{\alpha}\left(t\right)\right)^{*}y\left(t\right)+b_{x}\left(t,\bar{x}\left(t\right),\bar{\alpha}\left(t\right)\right)^{*}Y\left(t\right)\right.\nonumber \\
 & \left.-g_{x}\left(t,\bar{x}\left(t\right),\bar{\alpha}\left(t\right)\right)\right)dt+Y\left(t\right)dW\left(t\right)\textrm{ in }\left[0,T\right),\label{eq:backward_control}\\
y\left(T\right)= & -h_{x}\left(\bar{x}\left(T\right)\right).\nonumber 
\end{align}

Next, we deal with the (more technical) optimal control problem with
endpoint/state constraints. Consider the controlled stochastic differential
equation (\ref{eq:cont_sys}) with $\alpha\in\mathcal{U}_{2}$,
\[
\mathcal{U}_{2}=\left\{ \alpha\left(\cdot\right):\left[0,T\right]\rightarrow U:\alpha\left(\cdot\right)\in L_{\mathbb{F}}^{2}\left(0,T;H_{1}\right)\right\} ,
\]
where $U$ is a nonempty closed subset of the separable Hilbert space
$H_{1}$. Let $\mathcal{K}_{a}$ be a nonempty closed subset of $H$,
and $h:\Omega\times H\rightarrow\mathbb{R}$, $g^{j}:H\rightarrow\mathbb{R}$
($j=0,\ldots,n$). We associate to the control system (\ref{eq:cont_sys})
a Mayer cost functional, $\mathcal{J}_{M}\left(\cdot\right)$, given
by
\begin{equation}
\mathcal{J}_{M}\left(\alpha\left(\cdot\right),x_{0}\right)=E\left[h\left(x\left(T\right)\right)\right],\label{eq:Mayer_cost}
\end{equation}
along with the state constraint 
\begin{equation}
E\left[g^{0}\left(x\left(t\right)\right)\right]\leq0,\textrm{ for all }t\in\left[0,T\right],\label{eq:state_const}
\end{equation}
and the initial-final states constraints 
\begin{equation}
x_{0}\in\mathcal{K}_{a},\quad E\left[g^{j}\left(x\left(T\right)\right)\right]\leq0,\quad j=1,\ldots,n.\label{eq:ini_fin_const}
\end{equation}
In this case, the set of admissible controls (with initial datum $x_{0}$)
is given by
\[
\mathcal{U}_{ad}^{x_{0}}=\left\{ \alpha\in\mathcal{U}_{2}:\textrm{ the corresponding solution }x\left(t\right)\textrm{ of (\ref{eq:cont_sys}) satisfies (\ref{eq:state_const}) and (\ref{eq:ini_fin_const})}\right\} .
\]

We study the following optimal control problem: find $\left(\bar{x}_{0},\bar{\alpha}\left(\cdot\right)\right)\in\mathcal{K}_{a}\times\mathcal{U}_{ad}^{x_{0}}$
such that
\begin{equation}
\mathcal{J}_{M}\left(\bar{x}_{0},\bar{\alpha}\left(\cdot\right)\right)=\underset{\left(x_{0},\alpha\left(\cdot\right)\right)\in\mathcal{K}_{a}\times\mathcal{U}_{ad}^{x_{0}}}{\inf}\mathcal{J}_{M}\left(x_{0},\alpha\left(\cdot\right)\right).\label{eq:op_prob2}
\end{equation}
We note that it is possible to study the more general Bolza problem
using the previous formulation (see \cite{frankowska2020first} Section
1). 

Let 
\[
\varphi_{1}\left[t\right]=\varphi_{x}\left(t,\bar{x}\left(t\right),\bar{\alpha}\left(t\right)\right),\quad\varphi_{2}\left[t\right]=\varphi_{u}\left(t,\bar{x}\left(t\right),\bar{\alpha}\left(t\right)\right),
\]
where $\varphi$ can be either $a,b,f,g,$ or $h$(with appropriate
regularity). To analyze Problem (\ref{eq:op_prob2}), we will need
the auxiliary linearized stochastic control system
\begin{align}
dx_{1}\left(t\right)= & \left(Ax_{1}\left(t\right)+a_{1}\left[t\right]x_{1}\left(t\right)+a_{2}\left[t\right]\alpha_{1}\left(t\right)\right)dt\nonumber \\
 & +\left(b_{1}\left[t\right]x_{1}\left(t\right)+b_{2}\left[t\right]\alpha_{1}\left(t\right)\right)dW\left(t\right)\textrm{ in }\left(0,T\right],\nonumber \\
x_{1}\left(0\right)= & x_{1},\label{eq:Lin_sys}
\end{align}
along with its first-order adjoint equation 
\begin{align}
dy\left(t\right)= & -\left(A^{*}y\left(t\right)+a_{1}\left[t\right]^{*}y\left(t\right)+b_{1}\left[t\right]^{*}Y\left(t\right)\right)dt\nonumber \\
 & +d\psi\left(t\right)+Y\left(t\right)dW\left(t\right)\textrm{ in }\left(0,T\right],\nonumber \\
y\left(T\right)= & y_{T},\label{eq:Lin_adj}
\end{align}
where $y_{T}\in L_{\mathcal{F}_{T}}^{2}\left(\Omega;H\right)$ and
$\psi\in L_{\mathbb{F}}^{2}\left(\Omega;BV_{0}\left(\left[0,T\right];H\right)\right)$.

From the biological perspective, it is natural to ask if there is
a canonical cost function associated with the epigenetic landscape.
There are at least two options which seem realistic: to consider an
energy function or a minimal time problem. The former requires to
define an energy function corresponding to the given system. In general,
this function will depend on the detailed structure of the underlying
genetic regulatory network and the specific system. An alternative,
which relies on the assumption that the energetic cost spent in traversing
from a state (genetic expression profile), $A$, to another state,
$B$, is a convex function of the distance in the state space, e.g.,
the distance squared. As for the second option, it is compatible with
the hypothesis that the energy spent is a monotone function of the
time required to reach state $B$ starting at $A$. Both problems
are interesting from the mathematical perspective; they arise naturally,
especially the optimal time problem, in the context of epigenetic
therapy, where designing a control policy for the disease is crucial.

\section{Main Results\label{sec:Main_results}}
Our results consist of three parts:
\begin{enumerate}[label=\arabic*.]
\item Stochastic dynamics,
\item Invariant measures,
\item Stochastic control.
\end{enumerate}

\begin{theorem}
\label{thm:Thm_large_dev}Let $\sigma_{1}=\sigma_{2}=\sigma>0$ and
$u\left(0,x\right),v\left(0,x\right)\in C\left(\mathbb{T};\mathbb{R}\right)=C\left(\mathbb{T}\right)$.
The stochastic epigenetic system (\ref{eq:sys_epig}) has a unique
generalized solution, $u^{\sigma}\left(x,t\right)$, and the random
process $u^{\sigma}\left(t\right)=u^{\sigma}\left(t,\cdot\right)$
in the state space $C\left(\mathbb{T}\right)$ is a Markov-Feller
process. Furthermore,
\begin{enumerate}
\item The process $u^{\sigma}\left(t\right)$ has a unique normed stationary
measure, $\nu^{\sigma}$, in $C\left(\mathbb{T};\mathbb{R}^{2}\right)$
such that, for any Borel set $\Gamma\in C\left(\mathbb{T};\mathbb{R}^{2}\right)$
and any $u_{0}\in C\left(\mathbb{T};\mathbb{R}^{2}\right)$,
\[
P_{u_{0}}\left\{ \underset{T\rightarrow\infty}{\lim}\frac{1}{T}\int_{0}^{T}\chi_{\Gamma}\left(u^{\sigma}\left(t\right)\right)dt=\nu^{\sigma}\left(\Gamma\right)\right\} =1,
\]
where $\chi_{\Gamma}\left(u\right)$ is the indicator of the set $\Gamma\subset C\left(\mathbb{T};\mathbb{R}^{2}\right)$.
\item Let $\hat{\varphi}^{\left(1\right)},\ldots,\hat{\varphi}^{\left(4\right)}\in\mathbb{R}^{2}$
be the points where the epigenetic potential $F$ achieves its absolute
minimum, i.e., the points $\left(u_{1},v_{1}\right),$ $\left(u_{2},v_{2}\right),$
$\left(u_{3},v_{3}\right),$ and $\left(u_{4},v_{4}\right)$. Let
$\triangle_{k}=\det\left(\mathbf{H}\left(F\right)\left(\hat{\varphi}^{k}\right)\right)=4a_{k}^{2}>0$
, $k=1,\ldots4$, where $\mathbf{H}$ is the Hessian matrix. Then,
the measure $\nu^{\sigma}$ weakly converges as $\sigma\rightarrow0$
to the measure $\nu^{0}$ concentrated at the m points $\hat{\varphi}^{\left(1\right)},\ldots,\hat{\varphi}^{\left(4\right)}\in\mathbb{R}^{2}$
and 
\[
\nu^{0}\left(\hat{\varphi}^{\left(k\right)}\right)=\triangle_{k}^{-1}\left(\sum_{j=1}^{m}\triangle_{j}^{-1}\right)^{-1}.
\]
Since for Arabidopsis thaliana $a_{k}=1/c_{k}$, $k=1,\ldots,4$ (see
Section 3 \cite{cortes2022variational}), our previous result suggests
$a_{k}=1/\left(2\sqrt{c_{k}}\right)$.
\item $\varphi_{0}\left(x\right)\equiv\hat{\varphi}^{\left(k\right)}$,
$k\in\left\{ 1,\ldots,4\right\} $ is an asymptotically stable equilibrium
point of the deterministic epigenetic system, i.e., (\ref{eq:sys_epig})
with $\sigma_{1}=\sigma_{2}=0$. In addition, the phenotype (or cellular
type) regions associated to each $\hat{\varphi}^{\left(k\right)}$
are invariant. Assume that $D$ is a regular region (see Subsection
\ref{subsec:dev} below) in $C\left(\mathbb{T};\mathbb{R}^{2}\right)$
which contains $\varphi_{0}$ such that all the trajectories (of the
deterministic system) starting from $g\in D\cup\partial D$ tend to
$\varphi_{0}$ without leaving $D$. Let 
\begin{multline*}
U\left(\varphi\right)=\int_{\mathbb{T}}\left[\frac{1}{2}\sum_{k=1}^{2}d_{k}\left(\frac{d\varphi_{k}}{dx}\right)^{2}+F\left(\varphi\left(x\right)\right)\right]dx,\\
\varphi\left(x\right)=\left(\varphi_{1}\left(x\right),\varphi_{2}\left(x\right)\right)\in C\left(\mathbb{T};\mathbb{R}^{2}\right).
\end{multline*}
Let $\tau^{\sigma}=\tau_{D}^{\sigma}=\inf\left\{ t:\textrm{ }u^{\sigma}\left(t\right)\notin D\right\} $
be the first exist time of $u^{\sigma}\left(t\right)$ from $D$.
Then
\[
\underset{\sigma\rightarrow0^{+}}{\lim}\textrm{ }\sigma^{2}\ln E_{g}\tau^{\sigma}=2\underset{\varphi\in\partial D}{\min}\left(U\left(\varphi\right)-U\left(\varphi_{0}\right)\right),\textrm{ }g\in D,
\]
with the transition at the minimizer $\varphi\in\partial D$ of the previous expression; cf. the mountain pass points for the deterministic epigenetic model in Section 4.2 \cite{cortes2022variational}.
\end{enumerate}
\end{theorem}

Theorem \ref{thm:Thm_large_dev} follows from verifying that the epigenetic landscape introduced in Section \ref{sec:Stoch_model} satisfies the assumptions of the abstract framework developed in \cite{freidlin1988random}. The novelty therefore lies in the construction of the stochastic epigenetic model and the biological interpretation of the resulting invariant measures, transition paths, and exit-time asymptotics.

Now, we have the following set of assumptions:
\begin{assumption}
\label{assu:S1_p}Let $e\left(\cdot,\cdot,\cdot\right):\left[0,T\right]\times\Omega\times H\times U\rightarrow H$
and $b\left(\cdot,\cdot,\cdot\right):\left[0,T\right]\times\Omega\times H\times U\rightarrow\mathcal{L}_{2}^{0}$
satisfy:
\begin{enumerate}
\item For any $\left(x,\alpha\right)\in H\times U$, the functions $e\left(\cdot,x,\alpha\right):\left[0,T\right]\times\Omega\rightarrow H$
and $b\left(\cdot,x,\alpha\right)\rightarrow\mathcal{L}_{2}^{0}$
are $\mathbb{F}-$measurable,
\item For any $x\in H$ and a.e. $\left(t,\omega\right)\in\left(0,T\right)\times\Omega$,
the functions $e\left(t,x,\cdot\right):U\rightarrow H$ and $b\left(t,x,\cdot\right):U\rightarrow\mathcal{L}_{2}^{0}$
are continuous,
\item For any $\left(x_{1},x_{2},\alpha\right)\in H\times H\times U$ and
a.e. $\left(t,\omega\right)\in\left(0,T\right)\times\Omega$,
\[
\left\{ \begin{array}{l}
\left|e\left(t,x_{1},\alpha\right)-e\left(t,x_{2},\alpha\right)\right|_{H}+\left|b\left(t,x_{1},\alpha\right)-b\left(t,x_{2},\alpha\right)\right|_{\mathcal{L}_{2}^{0}}\leq C\left|x_{1}-x_{2}\right|_{H},\\
\left|e\left(t,0,\alpha\right)\right|_{H}+\left|b\left(t,0,\alpha\right)\right|_{\mathcal{L}_{2}^{0}}\leq C.
\end{array}\right.
\]
\end{enumerate}
\end{assumption}

\begin{assumption}
\label{assu:S2}Let $g\left(\cdot,\cdot,\cdot\right):\left[0,T\right]\times\Omega\times H\times U\rightarrow\mathbb{R}$
and $h\left(\cdot\right):\Omega\times H\rightarrow\mathbb{R}$ be
two functions satisfying:
\begin{enumerate}
\item For any $\left(x,\alpha\right)\in H\times U$, $g\left(\cdot,x,\alpha\right):\left[0,T\right]\times\Omega\rightarrow\mathbb{R}$
is $\mathbb{F}-$measurable and $h\left(x\right):\Omega\rightarrow\mathbb{R}$
is $\mathcal{F}_{T}-$measurable,
\item For any $x\in H$ and a.e. $\left(t,\omega\right)\in\left(0,T\right)\times\Omega$,
$g\left(t,x,\cdot\right):U\rightarrow\mathbb{R}$ is continuous,
\item For any $\left(x_{1},x_{2},\alpha\right)\in H\times H\times U$ and
a.e. $\left(t,\omega\right)\in\left(0,T\right)\times\Omega$,
\[
\left\{ \begin{array}{l}
\left|g\left(t,x_{1},\alpha\right)-g\left(t,x_{2},\alpha\right)\right|+\left|h\left(x_{1}\right)-h\left(x_{2}\right)\right|\leq C\left|x_{1}-x_{2}\right|_{H},\\
\left|g\left(t,0,\alpha\right)\right|+\left|h\left(0\right)\right|\leq C.
\end{array}\right.
\]
\end{enumerate}
\end{assumption}

\begin{assumption}
\label{assu:S4}The control region $U$ is a convex subset of a separable
Hilbert space, $\tilde{H}$, and the metric of $U$ is induced by
the norm of $\tilde{H}$, that is, $d\left(\alpha_{1},\alpha_{2}\right)=\left|\alpha_{1}-\alpha_{2}\right|_{\tilde{H}}$.
\end{assumption}

\begin{assumption}
\label{assu:S5_p}For a.e. $\left(t,\omega\right)\in\left(0,T\right)\times\Omega$,
the functions $e\left(t,\cdot,\cdot\right):H\times U\rightarrow H$,
$b\left(t,\cdot,\cdot\right):H\times U\rightarrow\mathcal{L}_{2}^{0}$,
$g\left(t,\cdot,\cdot\right):H\times U\rightarrow\mathbb{R}$, and
$h\left(\cdot\right):H\rightarrow\mathbb{R}$ are $C^{1}$. Furthermore,
for any $\left(x,\alpha\right)\in H\times U$ and a.e. $\left(t,\omega\right)\in\left(0,T\right)\times\Omega$,
we have
\[
\left\{ \begin{array}{l}
\left|e_{x}\left(t,x,\alpha\right)\right|_{\mathcal{L}\left(H\right)}+\left|b_{x}\left(t,x,\alpha\right)\right|_{\mathcal{L}\left(H;\mathcal{L}_{2}^{0}\right)}+\left|g_{x}\left(t,x,\alpha\right)\right|_{H}+\left|h_{x}\left(x\right)\right|_{H}\leq C,\\
\left|e_{\alpha}\left(t,x,\alpha\right)\right|_{\mathcal{L}\left(\tilde{H};H\right)}+\left|b_{\alpha}\left(t,x,\alpha\right)\right|_{\mathcal{L}\left(\tilde{H};\mathcal{L}_{2}^{0}\right)}+\left|g_{\alpha}\left(t,x,\alpha\right)\right|_{\mathcal{L}\left(\tilde{H}\right)}\leq C
\end{array}\right.
\]
\end{assumption}

\begin{theorem}
\label{thm:Thm_control1}Consider the controlled stochastic epigenetic
system (\ref{eq:cont_sys})-(\ref{eq:cont_epig}). Let Assumptions
\ref{assu:S1_p}-\ref{assu:S5_p} hold. Then,
\begin{enumerate}
\item For any $x_{0}\in L_{\mathcal{F}_{0}}^{p_{0}}\left(\Omega;H\right)$,
$p_{0}\geq2$, and $\alpha\left(\cdot\right)\in\mathcal{U}\left[0,T\right]$,
system (\ref{eq:cont_sys}) has a unique mild solution, $x\left(\cdot\right)\equiv x\left(\cdot;x_{0},\alpha\right)\in C_{\mathbb{F}}\left(\left[0,T\right];L^{p_{0}}\left(\Omega;H\right)\right)$,
such that
\[
\left|x\left(\cdot\right)\right|_{C_{\mathbb{F}}\left(\left[0,T\right];L^{p_{0}}\left(\Omega;H\right)\right)}\leq C\left(1+\left|x_{0}\right|_{L_{\mathcal{F}_{0}}^{p_{0}}\left(\Omega;H\right)}\right).
\]
Moreover, equation (\ref{eq:backward_control}) is well-posed in the
sense of transposition solution (see Definition 4.13 \cite{lu2021mathematical}). 
\item Let $\left(\bar{x}\left(\cdot\right),\bar{\alpha}\left(\cdot\right)\right)$
be an optimal pair for Problem (\ref{eq:op_cont_prob}) with $x_{0}\in L_{\mathcal{F}_{0}}^{2}\left(\Omega;H\right)$.
Then,
\begin{multline*}
\textrm{Re}\left\langle a_{u}\left(t,\bar{x}\left(t\right),\bar{\alpha}\left(t\right)\right)^{*}y\left(t\right)+b_{u}\left(t,\bar{x}\left(t\right),\bar{\alpha}\left(t\right)\right)^{*}Y\left(t\right)-g_{u}\left(t,\bar{x}\left(t\right),\bar{\alpha}\left(t\right)\right),\right.\\
\left.\alpha-\bar{\alpha}\left(t\right)\right\rangle _{\tilde{H}}\leq0,
\end{multline*}
a.e. $\left(t,\omega\right)\in\left[0,T\right]\times\Omega$, $\forall\alpha\in U$,
where $\left(y\left(\cdot\right),Y\left(\cdot\right)\right)$ is the
transposition solution of (\ref{eq:backward_control}).
\end{enumerate}
\end{theorem}

Next, we need some notation and ideas from set-valued analysis; see
Subsection \ref{subsec:Control} for details. In addition, we have
the following assumptions:
\begin{assumption}
\label{assu:As1_p}$e\left(\cdot,\cdot,\cdot,\cdot\right):\left[0,T\right]\times H\times H_{1}\times\Omega\rightarrow H$
and $b\left(\cdot,\cdot,\cdot,\cdot\right):\left[0,T\right]\times H\times H_{1}\times\Omega\rightarrow\mathcal{L}_{2}^{0}$
are two maps such that
\begin{enumerate}
\item For any $\left(x,\alpha\right)\in H\times H_{1}$, $e\left(\cdot,x,\alpha,\cdot\right):\left[0,T\right]\times\Omega\rightarrow H$
and $b\left(\cdot,x,\alpha,\cdot\right):\left[0,T\right]\times\Omega\rightarrow\mathcal{L}_{2}^{0}$
are $\mathcal{B}\left(\left[0,T\right]\right)\times\mathcal{F}$ measurable
and $\mathbb{F}-$adapted,
\item For any $\left(t,x,\omega\right)\in\left[0,T\right]\times H\times\Omega$,
$e\left(t,x,\cdot,\omega\right):H_{1}\rightarrow H$ and $b\left(t,x,\cdot,\omega\right):H_{1}\rightarrow\mathcal{L}_{2}^{0}$
are continuous and
\[
\begin{cases}
\left|e\left(t,x_{1},\alpha,\omega\right)-e\left(t,x_{2},\alpha,\omega\right)\right|_{H}+\left|b\left(t,x_{1},\alpha,\omega\right)-b\left(t,x_{2},\alpha,\omega\right)\right|_{\mathcal{L}_{2}^{0}}\leq\\
\qquad C\left|x_{1}-x_{2}\right|_{H}\quad\forall\left(t,x_{1},x_{2},\alpha,\omega\right)\in\left[0,T\right]\times H\times H\times H_{1}\times\Omega\\
\left|e\left(t,0,\alpha,\omega\right)\right|_{H}+\left|b\left(t,0,\alpha,\omega\right)\right|_{\mathcal{L}_{2}^{0}}\leq C,\quad\forall\left(t,\alpha,\omega\right)\in\left[0,T\right]\times H_{1}\times\Omega.
\end{cases}
\]
\end{enumerate}
\end{assumption}

\begin{assumption}
\label{assu:As2_p}For a.e. $\left(t,\omega\right)\in\left[0,T\right]\times\Omega$,
the functions $e\left(t,\cdot,\cdot,\omega\right):H\times H_{1}\rightarrow H$
and $b\left(t,\cdot,\cdot,\omega\right):H\times H_{1}\rightarrow\mathcal{L}_{2}^{0}$
are differentiable, and $\left(e_{x}\left(t,x,\alpha,\omega\right),e_{\alpha}\left(t,x,\alpha,\omega\right)\right)$
and $\left(b_{x}\left(t,x,\alpha,\omega\right),b_{\alpha}\left(t,x,\alpha,\omega\right)\right)$
are uniformly continuous with respect to $x\in H$ and $\alpha\in U$
(Fr{\'e}chet differentiability). There exists a nonnegative $\eta\in L_{\mathbb{F}}^{2}\left(0,T;\mathbb{R}\right)$
such that for a.e. $\left(t,\omega\right)\in\left[0,T\right]\times\Omega$
and for all $x\in H$ and $\alpha\in H_{1}$,
\[
\begin{cases}
\left|e\left(t,0,\alpha,\omega\right)\right|_{H}+\left|b\left(t,0,\alpha,\omega\right)\right|_{\mathcal{L}_{2}^{0}}\leq C\left(\eta\left(t,\omega\right)+\left|\alpha\right|_{H_{1}}\right),\\
\left|e_{x}\left(t,x,\alpha,\omega\right)\right|_{\mathcal{L}\left(H\right)}+\left|e_{\alpha}\left(t,x,\alpha,\omega\right)\right|_{\mathcal{L}\left(H_{1};H\right)}+\left|b_{x}\left(t,x,\alpha,\omega\right)\right|_{\mathcal{L}\left(H;\mathcal{L}_{2}^{0}\right)}\\
\hfill+\left|b_{\alpha}\left(t,x,\alpha,\omega\right)\right|_{\mathcal{L}\left(H_{1};\mathcal{L}_{2}^{0}\right)}\leq C.
\end{cases}
\]

\end{assumption}

\begin{assumption}
\label{assu:As3}The functional $h\left(\cdot,\omega\right):H\rightarrow\mathbb{R}$
is differentiable $\mathbb{P}-$a.s., and there exists an $\eta\in L_{\mathcal{F}_{T}}^{2}\left(\Omega\right)$
such that for any $x,\tilde{x}\in H$,
\[
\begin{cases}
h\left(x,\omega\right)\leq C\left(\eta\left(\omega\right)^{2}+\left|x\right|_{H}^{2}\right),\quad\left|h_{x}\left(0,\omega\right)\right|_{H}\leq C\eta\left(\omega\right),\quad\mathbb{P}-a.s.,\\
\left|h_{x}\left(x,\omega\right)-h_{x}\left(\tilde{x},\omega\right)\right|_{H}\leq C\left|x-\tilde{x}\right|_{H},\quad\mathbb{P}-a.s.
\end{cases}
\]
\end{assumption}

\begin{assumption}
\label{assu:As4}For $j=0,\ldots,n$, the functional $g^{j}:H\rightarrow\mathbb{R}$
is differentiable and for any $x,\tilde{x}\in H$,
\[
\left|g^{j}\left(x\right)\right|\leq C\left(1+\left|x\right|_{H}^{2}\right),\quad\left|g_{x}^{j}\left(x\right)-g_{x}^{j}\left(\tilde{x}\right)\right|_{H}\leq C\left|x-\tilde{x}\right|_{H}.
\]
\end{assumption}

Define the Hamiltonian 
\[
\mathbb{H}^{epig}\left(t,x,\alpha,p,q\right)=\left\langle p,a\left(t,x,\alpha\right)\right\rangle _{H}+\left\langle q,b\left(t,x,\alpha\right)\right\rangle _{\mathcal{L}_{2}^{0}},
\]
where $\left(t,x,\alpha,p,q\right)\in\left[0,T\right]\times H\times H_{1}\times H\times\mathcal{L}_{2}^{0}$
, with $a\left(\cdot\right)$, $b\left(\cdot\right)$ given by (\ref{eq:cont_epig}).
\begin{theorem}
\label{thm:Thm_control2}Consider the controlled stochastic epigenetic
system (\ref{eq:cont_sys})-(\ref{eq:cont_epig}). Let Assumptions
\ref{assu:As1_p}-\ref{assu:As4} hold. Then
\begin{enumerate}
\item For any $x_{0}\in H$ and $\alpha\left(\cdot\right)\in\mathcal{U}_{2}$,
system (\ref{eq:cont_sys}) has a unique mild solution, $x\left(\cdot\right)\equiv x\left(\cdot;x_{0},\alpha\right)\in L_{\mathbb{F}}^{2}\left(\Omega;C\left(\left[0,T\right];H\right)\right)$,
such that
\[
\left|x\left(\cdot\right)\right|_{L^{2}\left(\Omega;C\left(\left[0,T\right];H\right)\right)}\leq C\left(1+\left|x_{0}\right|_{H}\right).
\]
Moreover, for any $\alpha_{1}\in\mathcal{T}_{\Phi}\left(\bar{\alpha}\right)$
and $x_{1}\in T_{\mathcal{K}_{a}}^{b}\left(\bar{x}_{0}\right)$, (\ref{eq:Lin_sys})
has a unique solution, $x_{1}\left(\cdot\right)\in L_{\mathbb{F}}^{2}\left(\Omega;C\left(\left[0,T\right];H\right)\right)$,
and for $\psi\in L_{\mathbb{F}}^{2}\left(\Omega;BV_{0}\left(\left[0,T\right];H\right)\right)$,
(\ref{eq:Lin_adj}) has a unique transposition solution $\left(y,Y\right)\in D_{\mathbb{F}}\left(\left[0,T\right];L^{2}\left(\Omega;H\right)\right)\times L_{\mathbb{F}}^{2}\left(0,T;\mathcal{L}_{2}^{0}\right)$. 
\item Let $\left(\bar{x}\left(\cdot\right),\bar{\alpha}\left(\cdot\right),\bar{x}_{0}\right)$
be an optimal triple of Problem (\ref{eq:op_prob2}). If $E\left|g_{x}^{0}\left(\bar{x}\left(t\right)\right)\right|_{H}\neq0$
for any $t\in\mathcal{I}_{0}\left(\bar{x}\right)$, then there exists
$\lambda_{0}\in\left\{ 0,1\right\} $, $\lambda_{j}\geq0$ for $j\in\mathcal{I}\left(\bar{x}\right)$
and $\psi\in\left(\mathcal{Q}^{\left(1\right)}\right)^{-}$ with $\psi\left(0\right)=0$
verifying
\[
\lambda_{0}+\sum_{j\in\mathcal{I}\left(\bar{x}\right)}\lambda_{j}+\left|\psi\right|_{L_{\mathbb{F}}^{2}\left(\Omega;BV\left(0,T;H\right)\right)}\neq0,
\]
such that the corresponding transposition solution $\left(y\left(\cdot\right),Y\left(\cdot\right)\right)$
of the first order adjoint equation (\ref{eq:Lin_adj}) with $y\left(T\right)=-\lambda_{0}h_{x}\left(\bar{x}\left(T\right)\right)-\sum_{j\in\mathcal{I}\left(\bar{x}\right)}\lambda_{j}g_{x}^{j}\left(\bar{x}\left(T\right)\right)$
satisfies the variational inequality
\[
E\left\langle y\left(0\right),\nu\right\rangle _{H}+E\int_{0}^{T}\left\langle \mathbb{H}_{\alpha}^{epig}\left[t\right],v\left(t\right)\right\rangle _{H_{1}}dt\leq0,\textrm{ }\forall\nu\in\mathcal{T}_{\mathcal{K}_{a}}\left(\bar{x}_{0}\right),\textrm{ }\forall v\left(\cdot\right)\in\mathcal{T}_{\Phi}\left(\bar{\alpha}\right),
\]
where $\mathbb{H}_{\alpha}^{epig}\left[t\right]=\mathbb{H}_{\alpha}^{epig}\left(t,\bar{x}\left(t\right),\bar{\alpha}\left(t\right),y\left(t\right),Y\left(t\right)\right)$.
Furthermore, if $\mathcal{G}^{\left(1\right)}\cap\mathcal{Q}^{\left(1\right)}\cap\mathcal{E}^{\left(1\right)}\neq\emptyset$,
the above holds with $\lambda_{0}=1$.
\end{enumerate}
\end{theorem}

\subsection{Biological Interpretation of the Mathematical Results}\label{subsec:bio_interp}

Epigenetic landscapes dictate cellular identity across all kingdoms of life, from the developmental branching of \textit{Arabidopsis thaliana} to the malignant polarization of human immune cells \cite{sierra2025tumor}. In plant biology, these chromatin-mediated states are remarkably robust; yet prolonged environmental noise---such as chronic drought or temperature stress---forces \textit{Arabidopsis} to undergo stable, inheritable epigenetic modifications that form a transgenerational stress memory. When these same conserved, noise-responsive epigenetic mechanisms operate within the chaotic arena of the tumor microenvironment (TME), their adaptive nature is co-opted. Tumor-associated macrophages (TAMs) do not merely respond to this malignancy; they are epigenetically rewritten by its sustained, chronic stress (induced polarization). 

This molecular hijacking helps explain the enigmatic behavior of macrophages in cancer. Recruited initially by tumor necrosis, these cells should theoretically adopt a classically activated, antitumoral M1-like killer phenotype. Instead, they undergo aggressive functional reprogramming. Once trapped within the hypoxic core, they actively promote tumor development and invasion by orchestrating angiogenesis, driving the epithelial-to-mesenchymal transition (EMT), and degrading the extracellular matrix to forge physical escape routes from the extremely isolated and harsh region of the tumor. Crucially, this reprogrammed state fuels metastasis by assisting the formation and migration of highly aggressive circulating tumor cell (CTC) clusters. 

While traditional frameworks treat macrophage polarization or organ differentiation as a simple binary switch, the underlying biological landscape is a spatiotemporal continuum of biochemical forces. We model this dynamic process as a stochastic partial differential equation (SPDE) gradient flow across Waddington's epigenetic landscape \eqref{eq:sys_epig}. Within this unified framework, the mathematical results stated in Theorem \ref{thm:Thm_large_dev}, Theorem \ref{thm:Thm_control1}, and Theorem \ref{thm:Thm_control2} do not only establish abstract probabilistic properties; they also provide a rigorous dictionary for translating stochastic infinite-dimensional dynamics into concrete biological statements.

\begin{itemize}
    \item \textbf{The Ergodic Nature of Cell Phenotypes (Theorem \ref{thm:Thm_large_dev}, Part (i)):} 
    The existence of a unique, normed stationary measure $\nu^\sigma$ and the corresponding almost-sure convergence of the time-average trajectory:
    \[
    P_{u_{0}}\left\{ \underset{T\rightarrow\infty}{\lim}\frac{1}{T}\int_{0}^{T}\chi_{\Gamma}\left(u^{\sigma}\left(t\right)\right)dt=\nu^{\sigma}\left(\Gamma\right)\right\} =1
    \]
    reveals that cell population distributions possess statistically stable long-time behavior. Biologically, this means that tracking a single macrophage or plant cell over an infinitely long time horizon within a noisy environment is statistically equivalent to taking a snapshot of a vast population of cells at a single instance. The stationary measure $\nu^\sigma$ acts as the true phenotypic profile of the tissue, capturing the exact proportion of time a cell spends in the M1, M2, or distinct floral organ states under continuous environmental fluctuations.

    \item \textbf{The Zero-Noise Limit and Ideal Homeostasis (Theorem \ref{thm:Thm_large_dev}, Part (ii)):}
    Our result proves that as the magnitude of the stochastic noise tends to zero ($\sigma \rightarrow 0$), the stationary measure $\nu^\sigma$ weakly converges to a discrete measure $\nu^0$ concentrated strictly at the preferred phenotypes $\hat{\varphi}^{(1)}, \dots, \hat{\varphi}^{(4)}$. Biologically, this asymptotic concentration demonstrates how small environmental noise concentrates cell populations near preferred biological states. In the absence of external noise, cells are perfectly protected by Waddington's canalization, relaxing entirely into their normal biological identities---such as predictable floral organ sepals/petals or stable antitumoral M1 macrophages. Crucially, the weights of the limit measure:
    \[
    \nu^{0}\left(\hat{\varphi}^{\left(k\right)}\right)=\triangle_{k}^{-1}\left(\sum_{j=1}^{m}\triangle_{j}^{-1}\right)^{-1}
    \]
    are inversely proportional to the square root of the basin sizes ($a_k = 1/(2\sqrt{c_k})$). This reveals that the steady states that are biologically favored (possessing larger basins of attraction, $c_k$) correspond to potentials with shallower Hessians ($\triangle_k$), ensuring that the system naturally aggregates at these preferred functional profiles when noise is minimized.

    \item \textbf{Large Deviations and Malignant Reprogramming (Theorem \ref{thm:Thm_large_dev}, Part (iii)):}
    In a highly volatile environment like the TME, the noise magnitude $\sigma$ is strictly positive, allowing the system to deviate from these deterministic basins. Part (iii) utilizes the Freidlin-Wentzell large deviation framework to show that epigenetic mutations correspond to rare transitions between basins. The exit-time asymptotic:
    \[
    \underset{\sigma\rightarrow0^{+}}{\lim}\sigma^{2}\ln \mathbb{E}_{g}\tau^{\sigma}=2\underset{\varphi\in\partial D}{\min}\left(U\left(\varphi\right)-U\left(\varphi_{0}\right)\right)
    \]
    demonstrates that the transition of a macrophage from an initially healthy M1 basin $D$ into a pro-tumoral M2 state is a rare, noise-induced large deviation. The expected time $\mathbb{E}_g \tau^\sigma$ required for a cell to escape its healthy domain grows exponentially with the energy barrier $\Delta U = \min(U(\varphi) - U(\varphi_0))$ along the boundary $\partial D$. Within the hypoxic core, chronic biochemical stressors lower this potential barrier while amplifying environmental noise $\sigma$, exponentially accelerating the rate at which macrophages cross the mathematical ``mountain pass'' points. This results in the rapid, irreversible epigenetic rewriting of TAMs into pro-metastatic drivers that fuel EMT, angiogenesis, and CTC cluster migration.
\end{itemize}

\subsubsection*{Therapeutic Controls and Biological Restrictions}

Building upon this predictive framework, the remaining theorems translate these dynamics into an actionable therapeutic paradigm by analyzing a stochastic optimal control framework designed for clinical application.

\begin{itemize}
    \item \textbf{Optimal Evolution Biasing (Theorem \ref{thm:Thm_control1}):}
    Our formulation proves that optimal interventions can bias stochastic evolution. In this setting, the control functions model the dynamic deployment of ``epidrugs,'' engineered either to reverse existing malignant epigenetic modifications or to act as epigenetic stabilizers. Theorem \ref{thm:Thm_control1} establishes that even in a highly volatile, fluctuating environment, a systematically designed drug intervention sequence can actively reshape the underlying potential landscape. By forcing targeted shifts in the state trajectories, these controls redirect the cell population away from pro-tumoral phenotypes and guide them back into healthy, stable basins of attraction.
    
    \item \textbf{State-Constrained Therapies (Theorem \ref{thm:Thm_control2}):}
    Crucially, real-world clinical interventions cannot operate without strict boundary conditions. Theorem \ref{thm:Thm_control2} establishes that state-constrained controls successfully model epigenetic therapies under real biological restrictions. In a clinical setting, an unconstrained optimal control might recommend drug concentrations that are highly effective at reprogramming macrophages but fatally toxic to the patient. By introducing state constraints directly into the infinite-dimensional stochastic system, our framework enforces strict boundaries that safeguard healthy tissue, minimize systemic cytotoxicity, and ensure the treatment regimen operates within realistic, tolerable dosage windows. 
\end{itemize}

We summarize these structural mappings from mathematical proofs to phenotypic dynamics in the flowchart shown in Fig.~\ref{fig:biology_diag}.

\begin{figure}[H]
\begin{centering}
\includegraphics[scale=0.6]{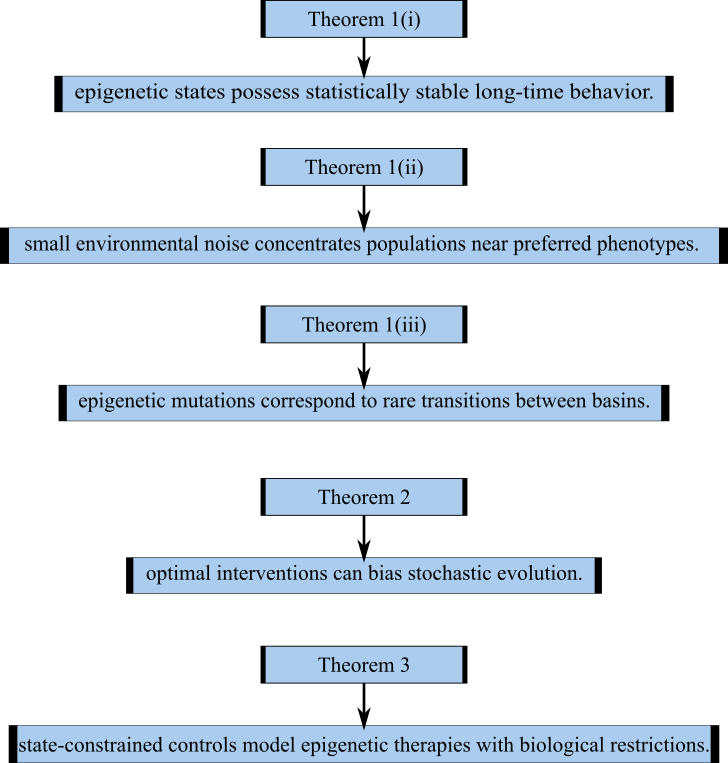}
\par\end{centering}
\caption{Flowchart mapping the mathematical results of Theorems 1, 2, and 3 to their corresponding biological interpretations. The sequential blocks illustrate how abstract behaviors---such as statistical stability, environmental noise effects, and rare basin transitions---translate to epigenetic modifications, population phenotypes, and optimal state-constrained therapeutic interventions. \label{fig:biology_diag}}
\end{figure}

\section{Preliminaries\label{sec:Preliminaries}}

For a thorough review of ergodic theory for stochastic PDE, see, e.g., \cite{da1996ergodicity,hairer2008ergodic}. For the the theory of
Large deviations, see \cite{freidlin1998random,varadhan2016large}.
See \cite{aastrom2012introduction,fleming2012deterministic,lu2021mathematical}
for a review of stochastic optimal control.

\subsection{Random Perturbations and Large Deviations\label{subsec:dev}}

We briefly recall the elements of Freidlin's theory that are needed in the proof of Theorem \ref{thm:Thm_large_dev}. Complete statements and proofs may be found in \cite{freidlin1988random}.

Consider the following stochastic evolution equation:
\begin{align}
\frac{\partial u_{k}^{\varepsilon}\left(t,x\right)}{\partial t}= & D_{k}\frac{\partial^{2}u_{k}^{\varepsilon}}{\partial x^{2}}+f_{k}\left(x,u_{1}^{\varepsilon},\ldots,u_{n}^{\varepsilon}\right)+\varepsilon\hat{\varsigma}_{k}\left(t,x\right),\label{eq:RDE}\\
u_{k}^{\varepsilon}\left(0,x\right)= & g_{k}\left(x\right),\nonumber 
\end{align}
$t>0$, $x\in\mathbb{T}$, $k=1,\ldots,n$. The perturbations $\hat{\varsigma}_{k}\left(t,x\right)$,
$k=1,\ldots,n$, are Gaussian fields which have independent values
for different $t$. Following \cite{freidlin1988random}, we consider a white noise perturbation:
$\hat{\varsigma}_{k}\left(t,x\right)=\partial^{2}\varsigma_{k}\left(t,x\right)/\partial t\partial x$,
where $\varsigma_{k}\left(t,x\right)$ are independent Brownian sheets
for different $k$. We assume that the functions $f_{k}\left(x,u\right)$,
$x\in\mathbb{T}$, $u\in\mathbb{R}^{n}$, are Lipschitz continuous.
Moreover, $g_{k}\left(x\right)\in C\left(\mathbb{T}\right)$ and $D_{k}>0$,
$k=1,\ldots,n$.

Under our assumptions, Theorem 1 in \cite{freidlin1988random} ensures
that (\ref{eq:RDE}) has a unique generalized solution. Furthermore,
the random process $u^{\varepsilon}\left(t\right)=u^{\varepsilon}\left(t,\cdot\right)$
in the state space $C\left(\mathbb{T}\right)$ is a Markov-Feller
process. 

Let $f\left(x,u\right)=\left(f_{1}\left(x,u\right),\ldots,f_{n}\left(x,u\right)\right)$,
where $u\in\mathbb{R}^{n}$ and $x\in\mathbb{T}$ is a parameter.
We call $f$ a potential field provided there exists a function, $F\left(x,u\right)$,
continuously differentiable in the variables $u\in\mathbb{R}^{n}$
and such that $f_{k}\left(x,u\right)=-\partial F\left(x,u\right)/\partial u_{k}$,
$x\in\mathbb{T}$, $u\in\mathbb{R}^{n}$, $k=1,\ldots,n$. Our main
interest is system (\ref{eq:RDE}) with the potential field $f\left(x,u\right)=-\nabla F\left(x,u\right)$.

Let $H^{1,2}$ be the Sobolev space of functions of $t\in\left[0,T\right]$
and $x\in\mathbb{T}$ with values in $\mathbb{R}^{n}$ which have
square integrable generalized derivatives of first order in $t$ and
second order in $x$.

By Theorem 6 in \cite{freidlin1988random}, the action functional
for the family of fields $u^{\varepsilon}\left(t,x\right)$, $0\leq t\leq T$,
$x\in\mathbb{T}$, in $C\left(\left[0,T\right]\times\mathbb{T};\mathbb{R}^{n}\right)$
as $\varepsilon\rightarrow0^{+}$ has the form $\varepsilon^{-2}S^{u}\left(\varphi\right)$
with

\[
S^{u}\left(\varphi\right)=\left\{ \begin{array}{l}
\frac{1}{2}\int_{\mathbb{T}}\int_{0}^{T}\sum_{k=1}^{n}\left|\frac{\partial\varphi_{k}}{\partial t}-D_{k}\frac{\partial^{2}\varphi_{k}}{\partial x^{2}}-f_{k}\left(x,\varphi\left(t,x\right)\right)\right|^{2}dtdx,\textrm{ }\varphi\in H^{1,2}\\
+\infty\textrm{ if }\varphi\in C\left(\left[0,T\right]\times\mathbb{T};\mathbb{R}^{n}\right)\backslash H^{1,2}.
\end{array}\right.
\]
Moreover, the functional $S^{u}\left(\varphi\right)$ is lower semicontinuous
on $C\left(\left[0,T\right]\times\mathbb{T};\mathbb{R}^{n}\right)$
and for every $s<\infty$, $g\in C\left(\mathbb{T};\mathbb{R}^{n}\right)$,
the set $\left\{ \varphi\in C\left(\left[0,T\right]\times\mathbb{T};\mathbb{R}^{n}\right):\varphi\left(0,x\right)=g\left(x\right),S^{u}\left(\varphi\right)\leq s\right\} $
is compact in $C\left(\left[0,T\right]\times\mathbb{T};\mathbb{R}^{n}\right)$.

We need the functional 
\begin{multline*}
V\left(g,h\right)=\inf\left\{ S^{u}\left(\varphi\right):\varphi\in C\left(\left[0,T\right]\times\mathbb{T};\mathbb{R}^{n}\right),\varphi\left(0,x\right)=g\left(x\right),\right.\\
\left.\varphi\left(T,x\right)=h\left(x\right),\textrm{ }T\geq0\right\} ,\quad g,h\in C\left(\mathbb{T};\mathbb{R}^{n}\right).
\end{multline*}

Assume that $\varphi_{0}\in C\left(\mathbb{T};\mathbb{R}^{n}\right)$
is an asymptotically stable equilibrium point of (\ref{eq:RDE}) with
$\varepsilon=0$ and let $D$ be a bounded open region in $C\left(\mathbb{T};\mathbb{R}^{n}\right)$
containing $\varphi_{0}$. The region $D\subset C\left(\mathbb{T};\mathbb{R}^{n}\right)$
is called regular if for every $\varphi\in\partial D$ there is a
twice continuously differentiable function, $h=h_{\varphi}\in C\left(\mathbb{T};\mathbb{R}^{n}\right)$,
such that $\varphi+th$ is an interior point of the complement of
$D\cup\partial D$ for all $t\geq0$ small enough. 

Let $\tau^{\varepsilon}=\tau_{D}^{\varepsilon}=\inf\left\{ t:u^{\varepsilon}\left(t\right)\notin D\right\} $
be the first exit time of $u\left(t\right)$ from $D$ and $V_{0}=\inf\left\{ V\left(\varphi_{0},\varphi\right):\varphi\in\partial D\right\} $.
Assume, that $D\in C\left(\mathbb{T};\mathbb{R}^{n}\right)$ is regular
and that $\varphi_{0}\in D$ is an asymptotically stable point of
(\ref{eq:RDE}) with $\varepsilon=0$. Furthermore, assume that every
trajectory of (\ref{eq:RDE}) with $\varepsilon=0$ starting at a
point $g\in D\cup\partial D$ does not leave $D$ for $t>0$ and tends
to $\varphi_{0}$ as $t\rightarrow\infty$. Then, by Theorem 8 in \cite{freidlin1988random}, for any $g\in D$
\[
\underset{\varepsilon\rightarrow0^{+}}{\lim}\textrm{ }\varepsilon^{2}\ln E_{g}\tau^{\varepsilon}=V_{0}.
\]
Moreover, if there is a unique $\varphi^{*}\in\partial D$ for which
$V\left(\varphi_{0},\varphi^{*}\right)=V_{0}$ , then the process
$u^{\varepsilon}\left(t\right)$ exists $D$ for the first time near
$\varphi^{*}$, that is, for any $\delta>0$ and $\varphi\in D$
\[
\underset{\varepsilon\rightarrow0^{+}}{\lim}P_{\varphi}\left\{ \underset{x\in\mathbb{T}}{\sup}\left|u_{\tau^{\varepsilon}}^{\varepsilon}\left(x\right)-\varphi^{*}\left(x\right)\right|>\delta\right\} =0.
\]

Now consider the functional $\mathscr{U}\left(\varphi\right)$
on $C\left(\mathbb{T};\mathbb{R}^{n}\right)$ taking finite values
on $H^{1}=H^{1}(\mathbb{T};\mathbb{R}^{n})$ and $+\infty$ on $C\left(\mathbb{T};\mathbb{R}^{n}\right)\backslash H^{1}$.
The functional $\mathscr{U}\left(\varphi\right)$ is called regular
if it is lower semicontinuous on $C\left(\mathbb{T};\mathbb{R}^{n}\right)$
equipped with the uniform convergence topology and the sets $\left\{ \varphi\in C\left(\mathbb{T};\mathbb{R}^{n}\right):\left\Vert \varphi\right\Vert \leq b,\mathscr{U}\left(\varphi\right)\leq a\right\} $
are compact in $C\left(\mathbb{T};\mathbb{R}^{n}\right)$ for any
$a,b\in\left(0,\infty\right)$.

For the convenience of the reader, we record in the following proposition the particular consequence of Theorem 9 in \cite{freidlin1988random} that will be used in the proof of Theorem \ref{thm:Thm_large_dev}.

\begin{proposition}[Theorem 9 in \cite{freidlin1988random}]
\label{prop:Exit_time}Assume that $\varphi_{0}\in C\left(\mathbb{T};\mathbb{R}^{n}\right)$
is an asymptotically stable equilibrium point of (\ref{eq:RDE}) with
$\varepsilon=0$. Let $B\left(x,u\right)=\left(B_{1}\left(x,u\right),\ldots,B_{n}\left(x,u\right)\right),$
$x\in\mathbb{T}$, $u\in C\left(\mathbb{T}\right)$, where
\[
B_{k}\left(x,u\right)=D_{k}\frac{d^{2}u}{dx^{2}}+f_{k}\left(x,u\right).
\]
Let a regular region, $D\subset C\left(\mathbb{T};\mathbb{R}^{n}\right)$,
be such that $\varphi_{0}\in D$ and every trajectory of (\ref{eq:RDE})
with $\varepsilon=0$ starting at a point $g\in D\cup\partial D$
does not leave $D$ for $t>0$ and tends to $\varphi_{0}$ as $t\rightarrow\infty$.
Furthermore, assume that there is a regular functional, $\mathscr{U}\left(\varphi\right)$,
and an operator, $L\left(\varphi\right)=\left(L_{1}\left(\varphi\right),\ldots,L_{n}\left(\varphi\right)\right)$,
$\varphi\in H^{1}$, such that
\begin{enumerate}
\item For $\varphi\in H^{2}$ the variational derivatives $\delta\mathscr{U}\left(\varphi\right)/\delta\varphi_{k}$,
$k=1,\ldots,n$, are defined and
\[
\left(\nabla\mathscr{U}\left(\varphi\right),L\left(\varphi\right)\right)=\int_{\mathbb{T}}\sum_{k=1}^{n}\frac{\delta\mathscr{U}}{\delta\varphi_{k}}\left(\varphi\left(x\right)\right)L_{k}\left(\varphi\left(x\right)\right)dx=0,\textrm{ }\varphi\in H^{2}.
\]
\item For the field $B\left(\varphi\right)=\left(B_{1}\left(\varphi\right),\ldots,B_{n}\left(\varphi\right)\right)$
we have
\[
B\left(\varphi\right)=-\nabla\mathscr{U}\left(\varphi\right)+L\left(\varphi\right),\textrm{ }\varphi\in H^{2}.
\]
\item For any $g\in H^{1}\cap\left(D\cup\partial D\right)$, there exists
a function $v\left(t,x\right)=\left(v_{1}\left(t,x\right),\ldots,v_{n}\left(t,x\right)\right)$,
$t>0$, $x\in\mathbb{T}$, such that
\[
\frac{\partial v_{k}\left(t,\cdot\right)}{\partial t}=-\frac{\delta\mathscr{U}\left(v\left(t,\cdot\right)\right)}{\delta v_{k}}-L_{k}\left(v_{k}\left(t,\cdot\right)\right),\textrm{ }t>0,k=1,\ldots,n,
\]
\[
v\left(0,x\right)=g\left(x\right),\textrm{ }\underset{t\rightarrow\infty}{\lim}\underset{x\in\mathbb{T}}{\sup}\left|v\left(t,x\right)-\varphi_{0}\left(x\right)\right|=0.
\]
\end{enumerate}
Then for $g\in H^{1}\cap\left(D\cup\partial D\right)$
\[
\inf\left\{ S^{u}\left(\varphi\right),\varphi\left(0,x\right)=\varphi_{0}\left(x\right),\varphi\left(T,x\right)=g\left(x\right),T>0\right\} =2\left(\mathscr{U}\left(g\right)-\mathscr{U}\left(\varphi_{0}\right)\right),
\]
and for any $g\in D$
\[
\underset{\varepsilon\rightarrow0^{+}}{\lim}\textrm{ }\varepsilon^{2}\ln E_{g}\tau^{\varepsilon}=2\underset{g\in\partial D}{\min}\left(\mathscr{U}\left(g\right)-\mathscr{U}\left(\varphi_{0}\right)\right).
\]
\end{proposition}

Following \cite{freidlin1988random}, equation (\ref{eq:RDE}) admits a unique invariant Gaussian measure, $\nu^{\varepsilon}$. Details of the construction and its properties can be found in \cite{freidlin1988random}. Assume that $f\left(u\right)=-\nabla F\left(u\right)$. Furthermore,
assume that $F\left(u\right)$ is three times continuously differentiable,
satisfies the inequality $F\left(u\right)>\alpha\left|u\right|+\beta$
for some $\alpha>0$, $\beta\in\left(-\infty,\infty\right)$ and attains
its absolute minimum at $m$ points $\hat{\varphi}^{\left(1\right)},\ldots,\hat{\varphi}^{\left(m\right)}\in\mathbb{R}^{n}$.
Moreover, let the critical points be nondegenerate, that is, $\triangle_{k}=\det\left(\partial^{2}F\left(\hat{\varphi}^{k}\right)/\partial u_{i}\partial u_{j}\right)\neq0$
for $k=1,\ldots m$. Then, by Theorem 5 in \cite{freidlin1988random},
the measure $\nu^{\varepsilon}$ weakly converges as $\varepsilon\rightarrow0$
to the measure $\nu^{0}$ concentrated at $m$ points $\hat{\varphi}^{\left(1\right)},\ldots,\hat{\varphi}^{\left(m\right)}\in\mathbb{R}^{n}$
and
\[
\nu^{0}\left(\hat{\varphi}^{\left(k\right)}\right)=\triangle_{k}^{-1}\left(\sum_{j=1}^{m}\triangle_{j}^{-1}\right)^{-1}.
\]

\subsection{Control Theory for Stochastic PDE\label{subsec:Control}}

As already stated, we must consider an optimal control problem with
endpoint/state constraints. Unfortunately, this problem is not well
understood up to now. Hence, we will start with a closely related
optimal control problem for which we can show a Pontriagyn-type maximum
principle. Afterward, we will study the case with endpoint/state constraints
using set-valued analysis and additional restrictions.

To apply the abstract stochastic control results of \cite{lu2021mathematical}, we additionally require the technical assumptions collected in Appendix \ref{sec:appendixA}.

Under Assumption \ref{assu:S1}, Proposition 12.1 of \cite{lu2021mathematical}
ensures that, for any $x_{0}\in L_{\mathcal{F}_{0}}^{p_{0}}\left(\Omega;H\right)$,
$p_{0}\geq2$, and $\alpha\left(\cdot\right)\in\mathcal{U}\left[0,T\right]$,
system (\ref{eq:cont_sys}) has a unique mild solution, $x\left(\cdot\right)\equiv x\left(\cdot;x_{0},\alpha\right)\in C_{\mathbb{F}}\left(\left[0,T\right];L^{p_{0}}\left(\Omega;H\right)\right)$,
such that
\[
\left|x\left(\cdot\right)\right|_{C_{\mathbb{F}}\left(\left[0,T\right];L^{p_{0}}\left(\Omega;H\right)\right)}\leq C\left(1+\left|x_{0}\right|_{L_{\mathcal{F}_{0}}^{p_{0}}\left(\Omega;H\right)}\right).
\]

Let Assumptions \ref{assu:S2} and \ref{assu:S1} hold. Then, by Theorem
4.19 of \cite{lu2021mathematical}, equation (\ref{eq:backward_control})
is well-posed in the sense of transposition solution (see Definition
4.13 of \cite{lu2021mathematical}); if we consider that the filtration
$\mathbf{F}$ is the natural one and $y_{T}\in L_{\mathcal{F}_{T}}^{p}$,
$p\in\left(1,2\right]$, then the solution is mild (see Section 4.2.1 of \cite{lu2021mathematical} for the notions of solutions and Section
4.2.2 of \cite{lu2021mathematical} for the case of natural filtration). 

We state Theorem 12.4 of \cite{lu2021mathematical} in the following
\begin{proposition}[Theorem 12.4 in \cite{lu2021mathematical}]
Let Assumptions \ref{assu:S2}-\ref{assu:S4} and \ref{assu:S1}-\ref{assu:s5}
hold. Let $\left(\bar{x}\left(\cdot\right),\bar{\alpha}\left(\cdot\right)\right)$
be an optimal pair for Problem (\ref{eq:op_cont_prob}) with $x_{0}\in L_{\mathcal{F}_{0}}^{2}\left(\Omega;H\right)$.
Then,
\begin{multline*}
\mathrm{Re}\left\langle a_{u}\left(t,\bar{x}\left(t\right),\bar{\alpha}\left(t\right)\right)^{*}y\left(t\right)+b_{u}\left(t,\bar{x}\left(t\right),\bar{\alpha}\left(t\right)\right)^{*}Y\left(t\right)-g_{u}\left(t,\bar{x}\left(t\right),\bar{\alpha}\left(t\right)\right),\right.\\
\left.\alpha-\bar{\alpha}\left(t\right)\right\rangle _{\tilde{H}}\leq0,
\end{multline*}
a.e. $\left(t,\omega\right)\in\left[0,T\right]\times\Omega$, $\forall\alpha\in U$,
where $\left(y\left(\cdot\right),Y\left(\cdot\right)\right)$ is the
transposition solution of (\ref{eq:backward_control}).
\end{proposition}

Now, we turn to the case of necessary optimality conditions for controlled
stochastic differential equations with control and state constraints.
Here, we follow closely the results presented in \cite{frankowska2020first}.

Under Assumption \ref{assu:As1}, Lemma 2.1 of \cite{frankowska2020first}
ensures that, for any $x_{0}\in H$ and $\alpha\left(\cdot\right)\in\mathcal{U}_{2}$,
system (\ref{eq:cont_sys}) has a unique mild solution, $x\left(\cdot\right)\equiv x\left(\cdot;x_{0},\alpha\right)\in L_{\mathbb{F}}^{2}\left(\Omega;C\left(\left[0,T\right];H\right)\right)$,
such that
\[
\left|x\left(\cdot\right)\right|_{L^{2}\left(\Omega;C\left(\left[0,T\right];H\right)\right)}\leq C\left(1+\left|x_{0}\right|_{H}\right).
\]

We review some basic ideas from set-valued analysis needed to study
(\ref{eq:op_prob2}); see \cite{aubin2009set} for details.

Let $Z$ be a Banach space and consider any subset $\mathcal{K}\subset Z$.
For $z\in\mathcal{K}$, the Clarke tangent cone $\mathcal{C}_{\mathcal{K}}\left(z\right)$
to $\mathcal{K}$ at $z$ is defined by
\[
\mathcal{C}_{\mathcal{K}}\left(z\right)=\left\{ v\in Z:\underset{\varepsilon\rightarrow0,y\in\mathcal{K},y\rightarrow z}{\lim}\frac{dist\left(y+\varepsilon v,\mathcal{K}\right)}{\varepsilon}=0\right\} ,
\]
where $dist\left(w,\mathcal{K}\right)=\underset{y\in\mathcal{K}}{\inf}\left|y-w\right|_{Z},$
$w\in Z$. Moreover, the adjacent cone $T_{\mathcal{K}}^{b}\left(z\right)$
to $\mathcal{K}$ at $z$ is given by
\[
T_{\mathcal{K}}^{b}\left(z\right)=\left\{ v\in Z:\underset{\varepsilon\rightarrow0^{+}}{\lim}\frac{dist\left(y+\varepsilon v,\mathcal{K}\right)}{\varepsilon}=0\right\} .
\]
$\mathcal{C}_{\mathcal{K}}\left(z\right)$ is a closed convex cone
in $Z$ and $\mathcal{C}_{\mathcal{K}}\left(z\right)\subset T_{\mathcal{K}}^{b}\left(z\right)$.
If $\mathcal{K}$ is convex, then $\mathcal{C}_{\mathcal{K}}\left(z\right)=T_{\mathcal{K}}^{b}\left(z\right)=\textrm{cl}\left\{ \alpha\left(\hat{z}-z\right)\textrm{ }:\textrm{ }z\geq0,\textrm{ }\hat{z}\in\mathcal{K}\right\} $,
where $\textrm{cl}\left\{ \cdot\right\} $ represents the closure
of the set. 

For a cone $\mathcal{K}$ in $Z$, the closed convex cone $\mathcal{K}^{-}=\left\{ \xi\in Z^{*}\textrm{ }:\textrm{ }\xi\left(z\right)\leq0\textrm{ for all }z\in\mathcal{K}\right\} $ is called the dual cone of $\mathcal{K}$.

Let $\left(\varTheta,\varSigma\right)$ be a measurable space and
$F:\Theta\rightsquigarrow Z$ be a set-valued map. The domain of $F$
is $\textrm{Dom}\left(F\right)=\left\{ \xi\in\Theta\textrm{ }:\textrm{ }F\left(\xi\right)\cap B\neq\emptyset\right\} $.
$F$ is measurable if $F^{-1}\left(B\right)=\left\{ \xi\in\Theta\textrm{ }:\textrm{ }F\left(\xi\right)\cap B\neq\emptyset\right\} \in\Sigma$
for any Borel set $B\in\mathcal{B}\left(Z\right)$.

Assume that $\left(\Theta,\Sigma,\mu\right)$ is a complete $\sigma-$finite
measure space and $F$ is a set-valued map from $\Theta$ to the separable
Banach space $\tilde{Z}$ with nonempty closed image. Then, by Lemma
2.3 \cite{frankowska2020first}, $F$ is measurable if and only if
its graph belongs to $\Sigma\otimes\mathcal{B}\left(\tilde{Z}\right)$. 

A map, $\zeta:\left(\Omega,\mathcal{F}\right)\rightsquigarrow Z$,
is a set valued random variable if it is measurable. A map, $\Psi:\left[0,T\right]\times\Omega\rightsquigarrow Z$,
is a measurable set-valued stochastic process if $\Psi$ is $\mathcal{B}\left[0,T\right]\otimes\mathcal{F}-$measurable;
$\Psi$ is adapted if $\Psi\left(t,\cdot\right)$ is $\mathcal{F}_{t}-$measurable
for all $t\in\left[0,T\right]$.

Let 
\[
\mathcal{G}=\left\{ B\in\mathcal{B}\left(\left[0,T\right]\right)\otimes\mathcal{F}\textrm{ }:\textrm{ }B_{t}\in\mathcal{F}_{t},\textrm{ }\forall t\in\left[0,T\right]\right\} ,
\]
where $B_{t}=\left\{ \omega\in\Omega\textrm{ }:\textrm{ }\left(t,\omega\right)\in B\right\} $.
Let $\textrm{m}$ be the Lebesgue measure on $\left[0,T\right]$.
We consider the completion of the measure space $\left(\left[0,T\right]\times\Omega,\mathcal{G},\mu=\textrm{m}\times P\right)$,
and due to Lemma 2.4 \cite{frankowska2020first}, we use the same
notation for the completion. 

Let $H$ be a separable Hilbert space. By Lemma 2.5 \cite{frankowska2020first},
a set-valued stochastic process, $F:\left[0,T\right]\times\Omega\rightsquigarrow H$,
is $\mathcal{B}\left(\left[0,T\right]\right)\otimes\mathcal{F}-$measurable
and $\mathbb{F}-$adapted if and only if $F$ is $\mathcal{G}-$measurable.

Let $\Phi$ be a set-valued stochastic process such that
\begin{enumerate}
\item $\Phi$ in $\mathcal{B}\left(\left[0,T\right]\right)\otimes\mathcal{F}-$measurable
and $\mathbb{F}-$adapted,
\item for a.e. $\left(t,\omega\right)\in\left[0,T\right]\times\Omega$,
$\Phi\left(t,\omega\right)$ is a nonempty closed convex cone in $H_{1}$.
\item $\Phi\left(t,\omega\right)\subset T_{U}^{b}\left(\bar{\alpha}\left(t,\omega\right)\right)$,
for a.e. $\left(t,\omega\right)\in\left[0,T\right]\times\Omega$.
\end{enumerate}
Define 
\[
\mathcal{T}_{\Phi}\left(\bar{\alpha}\right)=\left\{ \alpha\left(\cdot\right)\in L_{\mathbb{F}}^{2}\left(0,T;H_{1}\right)\textrm{ }:\textrm{ }\alpha\left(t,\omega\right)\in\Phi\left(t,\omega\right),\textrm{ a.e. }\left(t,\omega\right)\in\left[0,T\right]\times\Omega\right\} .
\]
$\mathcal{T}_{\Phi}\left(\bar{\alpha}\right)$ is a closed convex
cone in $L_{\mathbb{F}}^{2}\left(0,T;H_{1}\right)$. Moreover, since
$0\in\mathcal{T}_{\Phi}\left(\bar{\alpha}\right)$, $\mathcal{T}_{\Phi}\left(\bar{\alpha}\right)$
is nonempty. 

Under Assumption \ref{assu:As1}, for any $\alpha_{1}\in\mathcal{T}_{\Phi}\left(\bar{\alpha}\right)$
and $x_{1}\in T_{\mathcal{K}_{a}}^{b}\left(\bar{x}_{0}\right)$, (\ref{eq:Lin_sys})
has a unique solution, $x_{1}\left(\cdot\right)\in L_{\mathbb{F}}^{2}\left(\Omega;C\left(\left[0,T\right];H\right)\right)$;
see Section 3 of \cite{frankowska2020first}. Furthermore, under Assumptions
\ref{assu:As1}-\ref{assu:As2} and $\psi\in L_{\mathbb{F}}^{2}\left(\Omega;BV_{0}\left(\left[0,T\right];H\right)\right)$,
Lemma 3.5 of \cite{frankowska2020first} ensures that (\ref{eq:Lin_adj})
has a unique transposition solution $\left(y,Y\right)\in D_{\mathbb{F}}\left(\left[0,T\right];L^{2}\left(\Omega;H\right)\right)\times L_{\mathbb{F}}^{2}\left(0,T;\mathcal{L}_{2}^{0}\right)$;
see Definition 3.1 \cite{frankowska2020first}.

Let $T_{\mathcal{K}_{a}}\left(\bar{x}_{0}\right)$ be a nonempty closed
convex cone contained in $T_{\mathcal{K}_{a}}^{b}\left(\bar{x}_{0}\right)$.
Define
\begin{multline*}
\mathcal{G}^{\left(1\right)}=\left\{ x_{1}\left(\cdot\right)\in L_{\mathbb{F}}^{2}\left(\Omega;C\left(\left[0,T\right];H\right)\right)\textrm{ }:\textrm{ }\right.\\
\left.x_{1}\left(\cdot\right)\textrm{ solves (\ref{eq:Lin_sys}) with }\alpha_{1}\in\mathcal{T}_{\Phi}\left(\bar{\alpha}\right)\textrm{ and }x_{1}\in T_{\mathcal{K}_{a}}\left(\bar{x}_{0}\right)\right\} ,
\end{multline*}
\[
\mathcal{I}_{0}\left(\bar{x}\right)=\left\{ t\in\left[0,T\right]\textrm{ }:\textrm{ }E\left[g^{0}\left(\bar{x}\left(t\right)\right)\right]=0\right\} ,
\]
\[
\mathcal{I}\left(\bar{x}\right)=\left\{ j\in\left\{ 1,\ldots,n\right\} \textrm{ }:\textrm{ }E\left[g^{j}\left(\bar{x}\left(T\right)\right)=0\right]\right\} ,
\]
\[
\mathcal{Q}^{\left(1\right)}=\left\{ z\left(\cdot\right)\in L_{\mathbb{F}}^{2}\left(\Omega;C\left(\left[0,T\right];H\right)\right)\textrm{ }:\textrm{ }E\left\langle g_{x}^{0}\left(\bar{x}\left(t\right)\right),z\left(t\right)\right\rangle _{H}<0,\textrm{ }\forall t\in\mathcal{I}_{0}\left(\bar{x}\right)\right\} ,
\]
\[
\mathcal{E}^{\left(1,j\right)}=\left\{ z\left(\cdot\right)\in L_{\mathbb{F}}^{2}\left(\Omega;C\left(\left[0,T\right];H\right)\right)\textrm{ }:\textrm{ }E\left\langle g_{x}^{j}\left(\bar{x}\left(T\right)\right),z\left(T\right)\right\rangle _{H}<0\right\} ,\textrm{ }\forall j\in\mathcal{I}\left(\bar{x}\right),
\]
\[
\mathcal{E}^{\left(1\right)}=\bigcap_{j\in\mathcal{I}\left(\bar{x}\right)}\mathcal{E}^{\left(1,j\right)}.
\]

Since $\mathcal{T}_{\Phi}\left(\bar{\alpha}\right)$ and $\mathcal{T}_{\mathcal{K}_{a}}\left(\bar{x}_{0}\right)$
are nonempty convex cones, $\mathcal{G}^{\left(1\right)}$ is a nonempty
convex cone in $L_{\mathbb{F}}^{2}\left(\Omega;C\left(\left[0,T\right];H\right)\right)$.
Moreover, if $\mathcal{I}_{0}\left(\bar{x}\right)=\emptyset$ (resp.
$\mathcal{I}\left(\bar{x}\right)=\emptyset$), then $\mathcal{Q}^{\left(1\right)}=L_{\mathbb{F}}^{2}\left(\Omega;C\left(\left[0,T\right];H\right)\right)$
(resp. $\mathcal{E}^{\left(1\right)}=L_{\mathbb{F}}^{2}\left(\Omega;C\left(\left[0,T\right];H\right)\right)$).
By Lemma 3.4 \cite{frankowska2020first}, $\mathcal{Q}^{\left(1\right)}$
is an open convex cone in $L_{\mathbb{F}}^{2}\left(\Omega;C\left(\left[0,T\right];H\right)\right)$.

Define the Hamiltonian 
\[
\mathbb{H}\left(t,x,\alpha,p,q,\omega\right)=\left\langle p,a\left(t,x,\alpha,\omega\right)\right\rangle _{H}+\left\langle q,b\left(t,x,\alpha,\omega\right)\right\rangle _{\mathcal{L}_{2}^{0}},
\]
where $\left(t,x,\alpha,p,q,\omega\right)\in\left[0,T\right]\times H\times H_{1}\times H\times\mathcal{L}_{2}^{0}\times\Omega$.

Now, we state a first-order necessary optimality condition as presented
in Theorem 3.1 \cite{frankowska2020first}.
\begin{proposition}[Theorem 3.1 in \cite{frankowska2020first}]
Let Assumptions \ref{assu:As1}-\ref{assu:As2} and \ref{assu:As3}-\ref{assu:As4}
hold. Let $\left(\bar{x}\left(\cdot\right),\bar{\alpha}\left(\cdot\right),\bar{x}_{0}\right)$
be an optimal triple of Problem (\ref{eq:op_prob2}). If $E\left|g_{x}^{0}\left(\bar{x}\left(t\right)\right)\right|_{H}\neq0$
for any $t\in\mathcal{I}_{0}\left(\bar{x}\right)$, then there exists
$\lambda_{0}\in\left\{ 0,1\right\} $, $\lambda_{j}\geq0$ for $j\in\mathcal{I}\left(\bar{x}\right)$
and $\psi\in\left(\mathcal{Q}^{\left(1\right)}\right)^{-}$ with $\psi\left(0\right)=0$
verifying
\[
\lambda_{0}+\sum_{j\in\mathcal{I}\left(\bar{x}\right)}\lambda_{j}+\left|\psi\right|_{L_{\mathbb{F}}^{2}\left(\Omega;BV\left(0,T;H\right)\right)}\neq0,
\]
such that the corresponding transposition solution $\left(y\left(\cdot\right),Y\left(\cdot\right)\right)$
of the first order adjoint equation (\ref{eq:Lin_adj}) with $y\left(T\right)=-\lambda_{0}h_{x}\left(\bar{x}\left(T\right)\right)-\sum_{j\in\mathcal{I}\left(\bar{x}\right)}\lambda_{j}g_{x}^{j}\left(\bar{x}\left(T\right)\right)$
satisfies the variational inequality
\[
E\left\langle y\left(0\right),\nu\right\rangle _{H}+E\int_{0}^{T}\left\langle \mathbb{H}_{\alpha}\left[t\right],v\left(t\right)\right\rangle _{H_{1}}dt\leq0,\textrm{ }\forall\nu\in\mathcal{T}_{\mathcal{K}_{a}}\left(\bar{x}_{0}\right),\textrm{ }\forall v\left(\cdot\right)\in\mathcal{T}_{\Phi}\left(\bar{\alpha}\right),
\]
where $\mathbb{H}_{\alpha}\left[t\right]=\mathbb{H}_{\alpha}\left(t,\bar{x}\left(t\right),\bar{\alpha}\left(t\right),y\left(t\right),Y\left(t\right),\omega\right)$.
Furthermore, if $\mathcal{G}^{\left(1\right)}\cap\mathcal{Q}^{\left(1\right)}\cap\mathcal{E}^{\left(1\right)}\neq\emptyset$,
the above holds with $\lambda_{0}=1$.
\end{proposition}

\section{Proof of Results\label{sec:Proof_results}}
\begin{proof}[Proof of Theorem \ref{thm:Thm_large_dev}] Notice that, by Remark \ref{rem_F}, the mollified landscape $F$ satisfies the assumptions of Theorem 5 in \cite{freidlin1988random}. 

(i) and (ii) follow from the theory presented in subsection \ref{subsec:dev}.

Let 
\[
\Phi\left[u,v\right]=\int_{\mathcal{O}}L\left(u,v\right)dx,
\]
where $L=\frac{1}{2}\left(d_{1}\left|\partial u/\partial x\right|^{2}+d_{2}\left|\partial v/\partial x\right|^{2}\right)+F\left(u,v\right)$.
Then, we can write the deterministic part of (\ref{eq:sys_epig}),
i.e., $\sigma_{1}=\sigma_{2}=0$, as
\begin{align*}
\frac{\partial u}{\partial t}= & -\frac{\delta\Phi\left(u,v\right)}{\delta u},\\
\frac{\partial v}{\partial t}= & -\frac{\delta\Phi\left(u,v\right)}{\delta v},
\end{align*}
where $\delta$ stands for the variational derivative; this shows
the steepest descent (gradient flow) feature of the deterministic
system. Note that $\varphi_{0}\left(x\right)\equiv\hat{\varphi}^{\left(k\right)}$,
$k\in\left\{ 1,\ldots,4\right\} $ is an equilibrium point of the
deterministic epigenetic system. Moreover, $\det\left(\mathbf{H}\left(F\right)\left(\hat{\varphi}^{k}\right)\right)=4a_{k}^{2}>0$
, $k=1,\ldots4$; it follows that $\varphi_{0}$ is asymptotically
stable. On the other hand, Corollary 14.8 (b) in \cite{smoller2012shock}
ensures that the phenotype (or cellular type) regions associated to
each $\hat{\varphi}^{\left(k\right)}$ are invariant.

Since our field is potential, i.e., $\tilde{f}=-\partial F_{\varepsilon}\left(u,v\right)/\partial u$,
$\tilde{g}=-\partial F_{\varepsilon}\left(u,v\right)/\partial v$, then the field
$B\left(\varphi\right)$ is also potential (see Section 3 of \cite{freidlin1988random}):
\[
B\left(\varphi\right)=-\nabla U\left(\varphi\right).
\]
Note that $U\left(\varphi\right)$ is finite on $H^{1}$. Following
the ideas of Lemma 3 in \cite{freidlin1988random} we can show that
if $U\left(\varphi\right)$ is extended onto $C\left(\mathbb{T};\mathbb{R}^{2}\right)$
with $U\left(\varphi\right)=+\infty$ for $\varphi\in C\left(\mathbb{T};\mathbb{R}^{2}\right)\backslash H^{1}$,
then $U\left(\varphi\right)$ is lower semicontinuous in $C\left(\mathbb{T};\mathbb{R}^{2}\right)$.
Moreover, for any $a\in\left(0,\infty\right)$, the set $\Phi_{a}=\left\{ \varphi\in C\left(\mathbb{T};\mathbb{R}^{2}\right)\textrm{ }:\textrm{ }U\left(\varphi\right)\leq a\right\} $
is compact in $C\left(\mathbb{T};\mathbb{R}^{2}\right)$ (Arzela-Ascoli
theorem). Hence, $U\left(\varphi\right)$ is a regular functional
and (iii) follows from Proposition \ref{prop:Exit_time}.
\end{proof}
\begin{proof}[Proof of Theorem \ref{thm:Thm_control1}] Both $\tilde{f}$ and $\tilde{g}$
are smooth with bounded derivatives of all orders since we are considering
the mollified version of $F$. The corresponding (autonomous) Nemytskii
operators on $H^{1}$ are Lipschitz and continuously (Fr{\'e}chet) differentiable;
see Theorem 1.4 \cite{efendiev2023linear}. The result is also valid
if we consider $H^{k}$, $k>1/2$; see \cite{ambrosetti1995primer,appell1981implicit,appell1990nonlinear}
for details. The $L^{2}$ case is more complicated, and the associated
Nemytskii operators are only Gateaux differentiable; see Theorem 2.7 \cite{ambrosetti1995primer}. Then, (i) and (ii) follow if we consider
in addition Assumptions \ref{assu:S1_p}-\ref{assu:S5_p}; see Subsection
\ref{subsec:Control}.
\end{proof}
\begin{proof}[Proof of Theorem \ref{thm:Thm_control2}] We proceed similarly to
the proof of Theorem \ref{thm:Thm_control1} for the regularity of
$\tilde{f}$ and $\tilde{g}$. Then, (i) and (ii) follow if we consider
in addition Assumptions \ref{assu:As1_p}-\ref{assu:As4}. 
\end{proof}

\section{Numerical Simulations\label{sec:Numerical_simulations}}

We now present some numerical simulations showing the dynamics of
our stochastic system. In particular, the exit from a basin of attraction
and the evolution afterward. Our results are based on the ideas presented
in Chapter 10 of \cite{lord2014introduction}.

We use the finite difference method for our stochastic epigenetic
system with homogeneous Neumann boundary conditions on $\mathcal{O}=\left(0,1\right)$.
Set $\sigma_{1}=\sigma_{2}=\sigma>0$ and let $W^{1}\left(t\right)$
and $W^{2}\left(t\right)$ be two independent $Q-$Wiener processes
on $L^{2}\left(\mathcal{O}\right)$ with kernel $q\left(x,y\right)=\exp\left(-\left|x-y\right|/l\right)$
for a correlation length $l>0$; the white noise case is more difficult
to handle, and we will avoid it for our illustrative purposes.

Consider the grid points $x_{j}=jh$ for $h=1/J$ and $j=0,\ldots,J$.
Let $\boldsymbol{u}_{J}\left(t\right)$ and $\boldsymbol{v}_{J}\left(t\right)$
be the finite difference approximations to $\left[u\left(t,x_{1}\right),\ldots,u\left(t,x_{J-1}\right)\right]^{\mathrm{T}}$
and $\left[v\left(t,x_{1}\right),\ldots,v\left(t,x_{J-1}\right)\right]^{\mathrm{T}}$,
respectively, resulting from the centered difference approximation
$A^{N}$ of the (negative of the) Laplacian (see eq. (3.51) in \cite{lord2014introduction}).
That is, $\boldsymbol{u}_{J}\left(t\right)$ and $\boldsymbol{v}_{J}\left(t\right)$
are the solution of 
\begin{align*}
d\boldsymbol{u}_{J}= & \left[-d_{1}A^{N}\boldsymbol{u}_{J}+\tilde{f}\left(\boldsymbol{u}_{J},\boldsymbol{v}_{J}\right)\right]dt+\sigma\boldsymbol{W}_{J}^{1}\left(t\right),\\
d\boldsymbol{v}_{J}= & \left[-d_{2}A^{N}\boldsymbol{v}_{J}+\tilde{g}\left(\boldsymbol{u}_{J},\boldsymbol{v}_{J}\right)\right]dt+\sigma\boldsymbol{W}_{J}^{2}\left(t\right),
\end{align*}
with $\boldsymbol{u}_{J}\left(0\right)=\left[u_{0}\left(x_{1}\right),\ldots,u_{0}\left(x_{J-1}\right)\right]^{\mathrm{T}}$,
$\boldsymbol{v}_{J}\left(0\right)=\left[v_{0}\left(x_{1}\right),\ldots,v_{0}\left(x_{J-1}\right)\right]^{\mathrm{T}}$,
and $\boldsymbol{W}_{J}^{k}=\left[W^{k}\left(t,x_{1}\right),\ldots,W^{k}\left(t,x_{J-1}\right)\right]^{\mathrm{T}}$,
$k=1,2.$ To discretize in time, we apply the semi-implicit Euler-Maruyama
method (see eq. (8.121) in \cite{lord2014introduction}) with time
step $\triangle t>0$, which gives the approximations $\boldsymbol{u}_{J,n}$
to $\boldsymbol{u}_{J}\left(t_{n}\right)$ and $\boldsymbol{v}_{J,n}$
to $\boldsymbol{v}_{J}\left(t_{n}\right)$ at $t_{n}=n\triangle t$
defined by
\begin{align*}
\boldsymbol{u}_{J,n+1}= & \left(I+\triangle t\textrm{ }d_{1}A^{N}\right)^{-1}\left[\boldsymbol{u}_{J,n}+\tilde{f}\left(\boldsymbol{u}_{J,n},\boldsymbol{v}_{J,n}\right)\triangle t+\sigma\triangle\boldsymbol{W}_{n}^{1}\right],\\
\boldsymbol{v}_{J,n+1}= & \left(I+\triangle t\textrm{ }d_{2}A^{N}\right)^{-1}\left[\boldsymbol{v}_{J,n}+\tilde{g}\left(\boldsymbol{u}_{J,n},\boldsymbol{v}_{J,n}\right)\triangle t+\sigma\triangle\boldsymbol{W}_{n}^{2}\right],
\end{align*}
with $\boldsymbol{u}_{J,0}=\boldsymbol{u}_{J}\left(0\right)$, $\boldsymbol{v}_{J,0}=\boldsymbol{v}_{J}\left(0\right)$,
and $\triangle\boldsymbol{W}_{n}^{k}=\boldsymbol{W}_{J}^{k}\left(t_{n+1}\right)-\boldsymbol{W}_{J}^{k}\left(t_{n}\right)$,
$k=1,2$. Furthermore, $\boldsymbol{W}_{J}^{k}\left(t\right)\sim N\left(\mathbf{0},tC\right)$,
where $C$ is the matrix with entries $q\left(x_{i},x_{j}\right)$
for $i,j=1,\ldots,J-1$. We use the circulant embedding method to
generate the increments $\triangle\boldsymbol{W}_{n}^{k}$ (see Algorithms
6.9 and 10.7 in \cite{lord2014introduction}). For the stability and
convergence analysis of this method, see Chapter 10 in \cite{lord2014introduction}.

We use Python for our implementation and the function gaussian\_filter
of scipy.ndimage to smooth out the epigenetic potential (\ref{eq:Pot_F}).
Fig. (\ref{fig:Color_organs}) shows the colors corresponding to the
four organs in the flower Arabidopsis thaliana. Fig (\ref{fig:Exit_1})
shows the epigenetic landscape and the system's evolution (sample
path) for different values of $t$ (see also the animation in the
online version); the white dot at $t=0.00$ corresponds to the initial
condition. We compute the $L^{2}-$average (Avg) of the sample path
at each time:
\begin{equation}
Avg\textrm{ }u\left(t\right)=\sqrt{\frac{1}{\left|\mathcal{O}\right|}\int_{\mathcal{O}}\left|u\left(t,x\right)\right|^{2}dx},\textrm{ }Avg\textrm{ }v\left(t\right)=\sqrt{\frac{1}{\left|\mathcal{O}\right|}\int_{\mathcal{O}}\left|v\left(t,x\right)\right|^{2}dx},\label{eq:avg}
\end{equation}
where $\left|\mathcal{O}\right|$ is the length of $\mathcal{O}$. We use the trapezoidal method for the integrals and the initial conditions
are constant functions.

For mean escape times, the leading-order exponential scaling is determined by the Freidlin--Wentzell quasipotential. A more accurate asymptotic approximation is obtained by incorporating the Eyring--Kramers prefactor, which accounts for fluctuations about the stable equilibrium and the relevant saddle point. Following the infinite-dimensional Eyring--Kramers framework of \cite{brooks2026eyring}, this prefactor is expressed as the ratio of the corresponding Fredholm determinants. For the present system, we obtain $C_{\text{EK}} \approx 485$.

We computed the approximate Eyring--Kramers prefactor $C_{\text{EK}}$ as follows. The Gaussian-smoothed potential $F_{\varepsilon}$ was generated numerically using the same Gaussian filter employed in the stochastic simulations. The metastable minimum and the relevant saddle separating the first two basins were located numerically. Their Hessian matrices were approximated using second-order centered finite differences, from which the determinants and the unique negative eigenvalue at the saddle were computed. Substituting these quantities into the finite-dimensional Eyring--Kramers formula
\[
C_{\text{EK}} = \frac{2\pi}{|\lambda_{-}|} \sqrt{\frac{|\det H_{s}|}{|\det H_{m}|}}.
\]
yields $C_{\text{EK}} \approx 485$. Since our stochastic model is an SPDE, this value should be interpreted as the prefactor associated with the finite-dimensional Gaussian-smoothed landscape used in the numerical simulations rather than the full infinite-dimensional SPDE prefactor. Incorporating this prefactor substantially improves the agreement between the asymptotic prediction and the numerical escape times reported in Table~\ref{tab1}.

\begin{table}[h]
\centering
\caption{Escape time estimates from sepals to petals for different $\sigma$ values, using Freidlin--Wentzell theory and Eyring--Kramers law. The fourth column shows the corresponding numerically observed values from the simulations in Figs.~\ref{fig:Trans_sep} and \ref{fig:Exit_1}.}\label{tab1}
\begin{tabular}{l r r r}
\toprule
$\sigma$ & Freidlin--Wentzell & Eyring--Kramers & Numerical simulations \\
\midrule
0.030    &   3.300           & $1.60 \times 10^3$ & $1.4 \times 10^3$   \\
0.020    &  14.675           & $7.12 \times 10^3$ & $6.4 \times 10^3$   \\
0.014    & 240.380           & $1.17 \times 10^5$ & $1.1 \times 10^5$   \\
\bottomrule
\end{tabular}
\end{table} 

The observed escape times are in close agreement with the Eyring--Kramers approximation and substantially improve upon the leading-order Freidlin--Wentzell estimate, illustrating the importance of the geometric prefactor for moderate noise intensities.

Fig (\ref{fig:Exit_1}) shows that, due to the effect of the noise,
the system eventually exits the local minima, traversing the epigenetic
landscape in the spatial order that corresponds to the correct architecture
of the flower, that is, following the observed geometrical features
of the meristem ($sepals\rightarrow petals\rightarrow stamens\rightarrow carpels$).
We remark that both the depths and the locations of the minima of
the basins of attraction play a crucial role in describing the correct
epigenetic dynamics of a biological system such as Arabidopsis thaliana.
To see this, Fig \ref{fig:Exit_2} displays the system's evolution
for a different arrangement of the epigenetic potential, which is
not the one expected from observed mutations of Arabidopsis thaliana.
Furthermore, these numerical simulations show the significance of
our rigorous results and how they can help to study other biological
systems (see Theorem \ref{thm:Thm_large_dev}).

Figs (\ref{fig:Dist1})-(\ref{fig:Dist2}) show the distribution of
Avg $u\left(t\right)$ and Avg $v\left(t\right)$ (see (\ref{eq:avg}))
of a sample path up to time $t=T$ with the same initial condition
as in Fig (\ref{fig:Exit_1}) and for different values of $\sigma$
(see also the animations in the online version). Notice that, as the
value of $\sigma$ is halved, so is the standard deviation of the
distribution. Moreover, the mean of the distributions is equal to
the (Avg of the) minimum of the basin of attraction (sepals in this
case). Both observations agree with our results in Theorem \ref{thm:Thm_large_dev}.

\begin{figure}[H]
\begin{centering}
\includegraphics[scale=0.6]{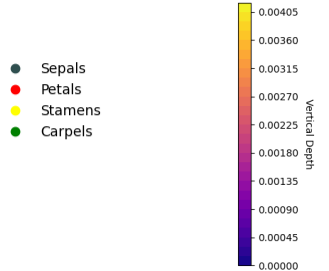}
\par\end{centering}
\caption{\label{fig:Color_organs}Colors corresponding to the four organs in
the flower Arabidopsis thaliana (left) and vertical depth in the epigenetic (Waddington) landscape (right) for our different contour plots.}

\end{figure}

\begin{figure}[H]
\begin{centering}
\includegraphics[scale=0.6]{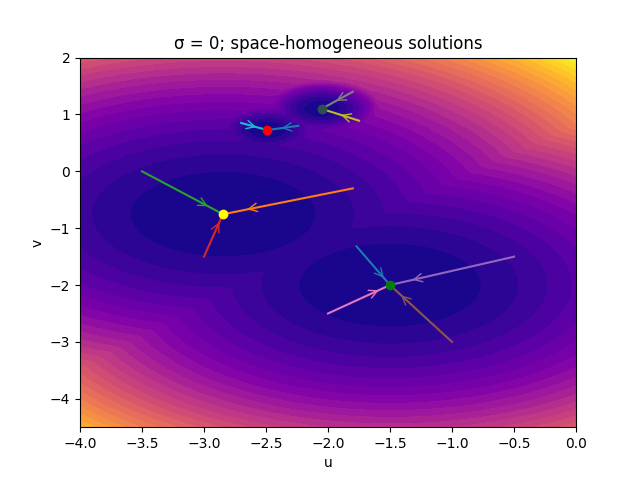}
\par\end{centering}
\caption{\label{fig:sigma_zero}In deterministic dynamics ($\sigma=0$), each phenotype region coincides with an invariant neighborhood of one minimum of the epigenetic landscape. These regions lie within the corresponding basin of attraction, and every deterministic trajectory that starts inside remains there for all positive time. The figure shows the steepest-descent trajectory used by our numerical method to the nearest local minimum, where the solution remains, illustrating the expected behavior of the underlying gradient flow. }

\end{figure}

\begin{figure}[H]
\begin{centering}
\begin{tabular}{cc}
\includegraphics[scale=0.18]{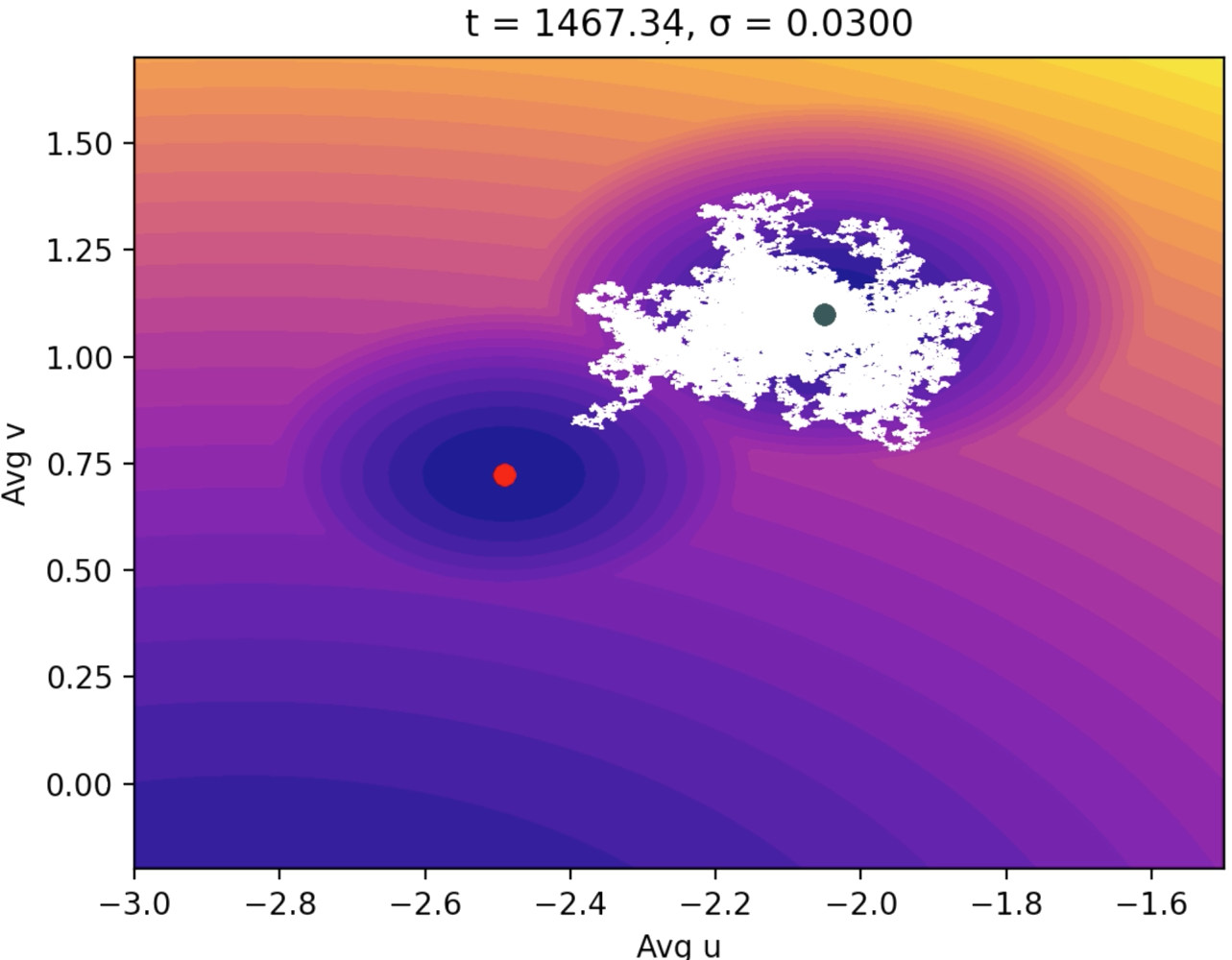} & \includegraphics[scale=0.18]{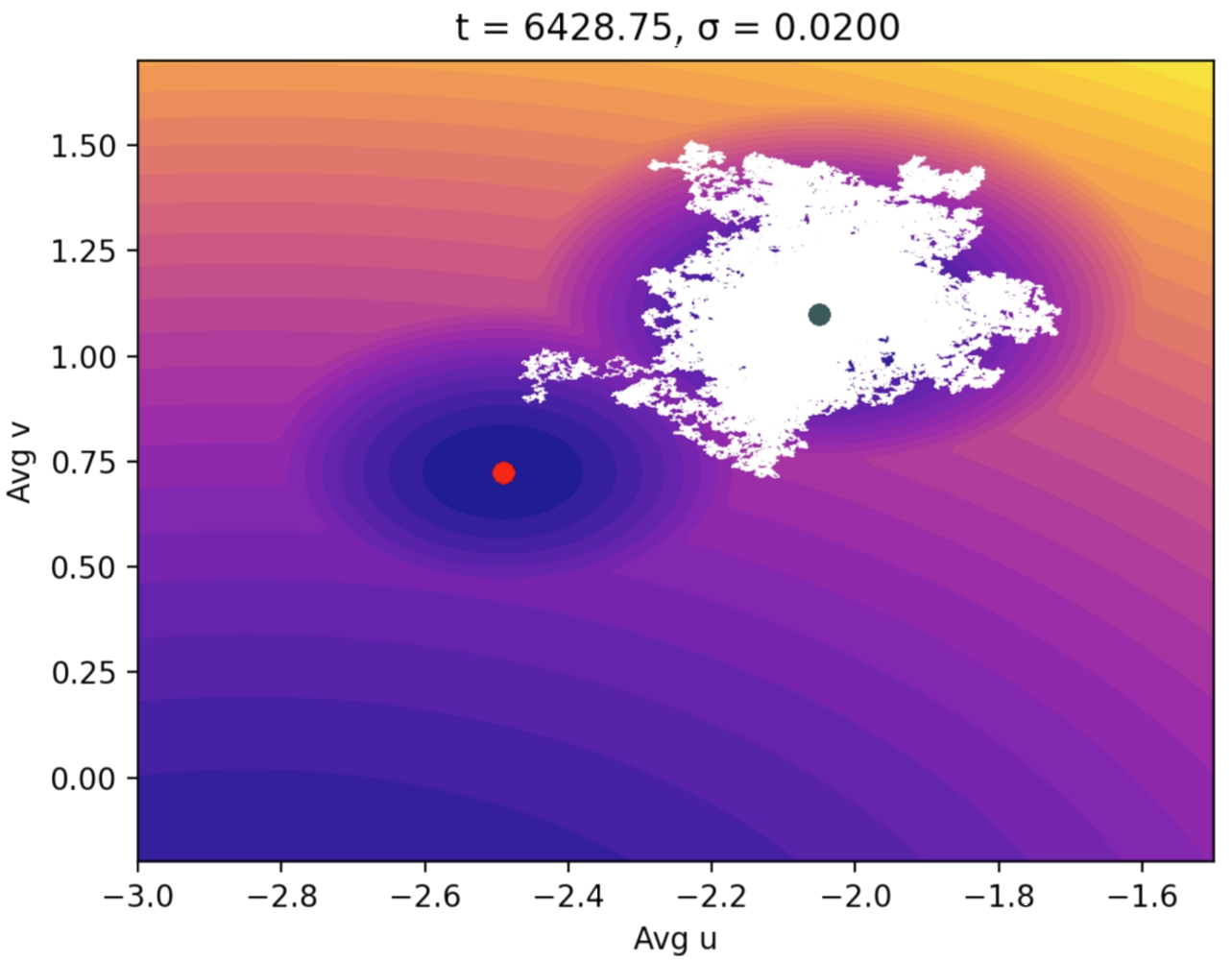}\tabularnewline
\end{tabular}
\par\end{centering}
\caption{Transitions (escape) from sepals to petals for different $\sigma$ values.\label{fig:Trans_sep}}

\end{figure}

\begin{figure}[H]
\begin{centering}
\begin{tabular}{cc}
\includegraphics[scale=0.14]{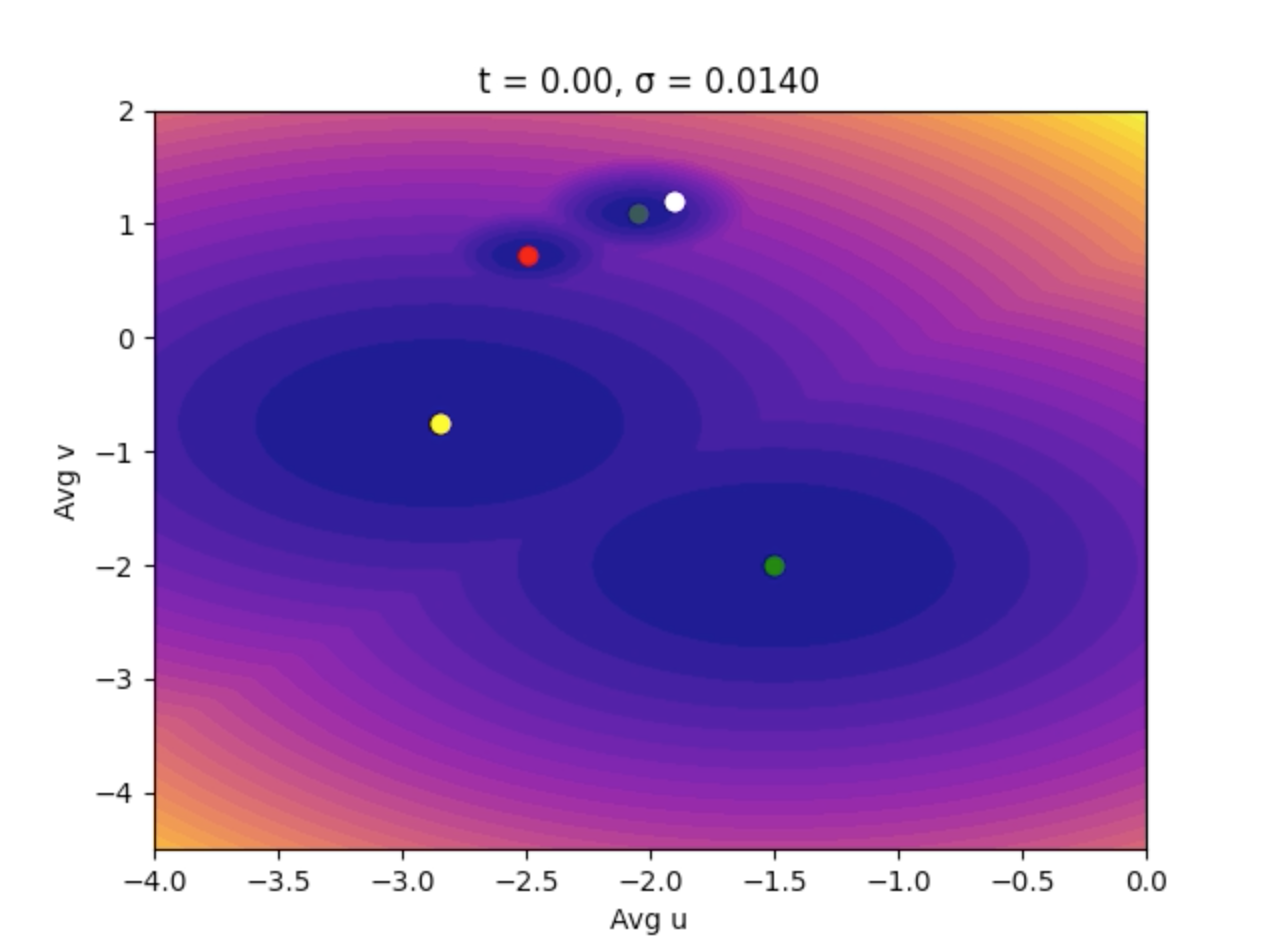} & \includegraphics[scale=0.14]{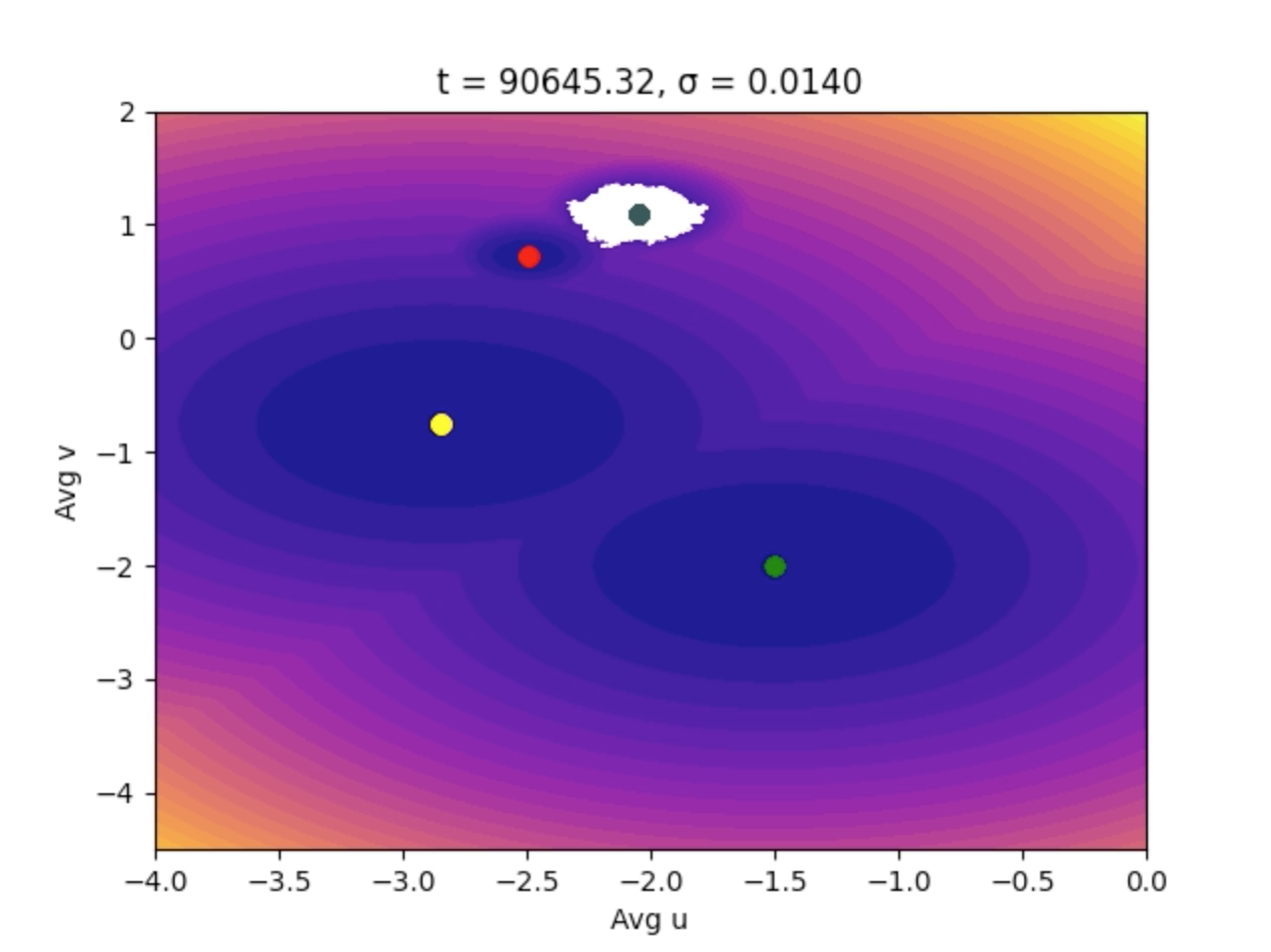}\tabularnewline
\includegraphics[scale=0.14]{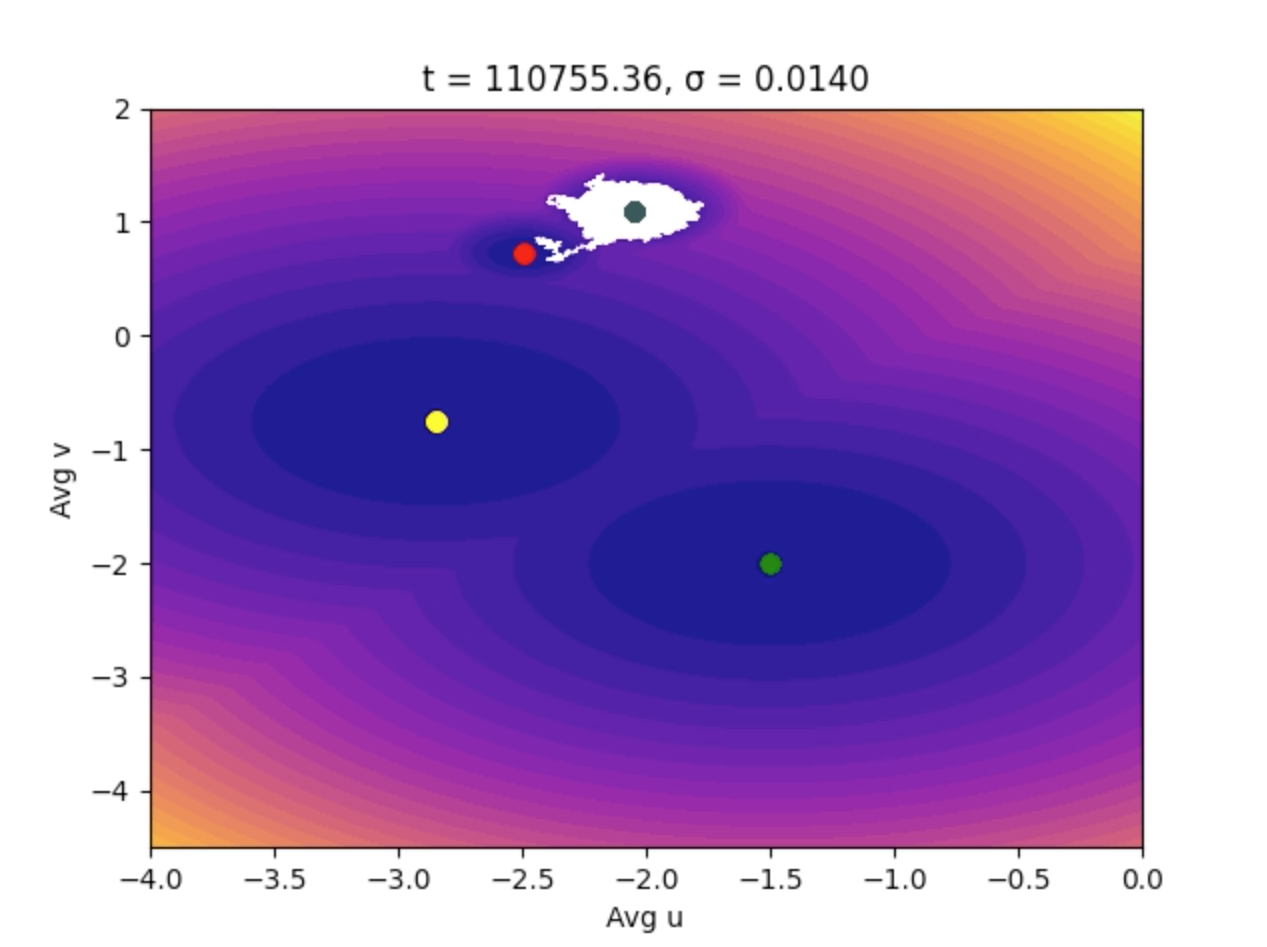} & \includegraphics[scale=0.14]{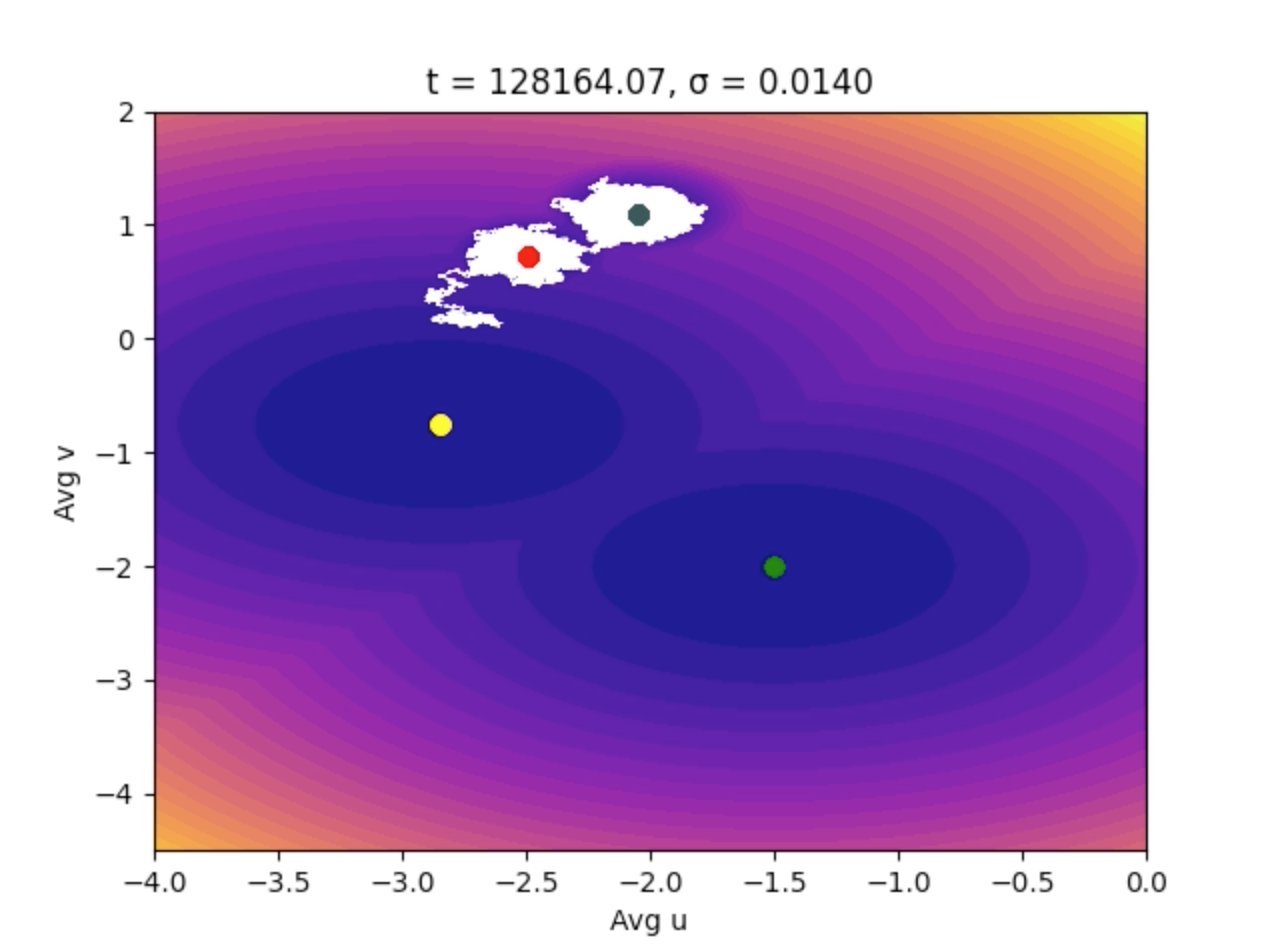}\tabularnewline
\includegraphics[scale=0.14]{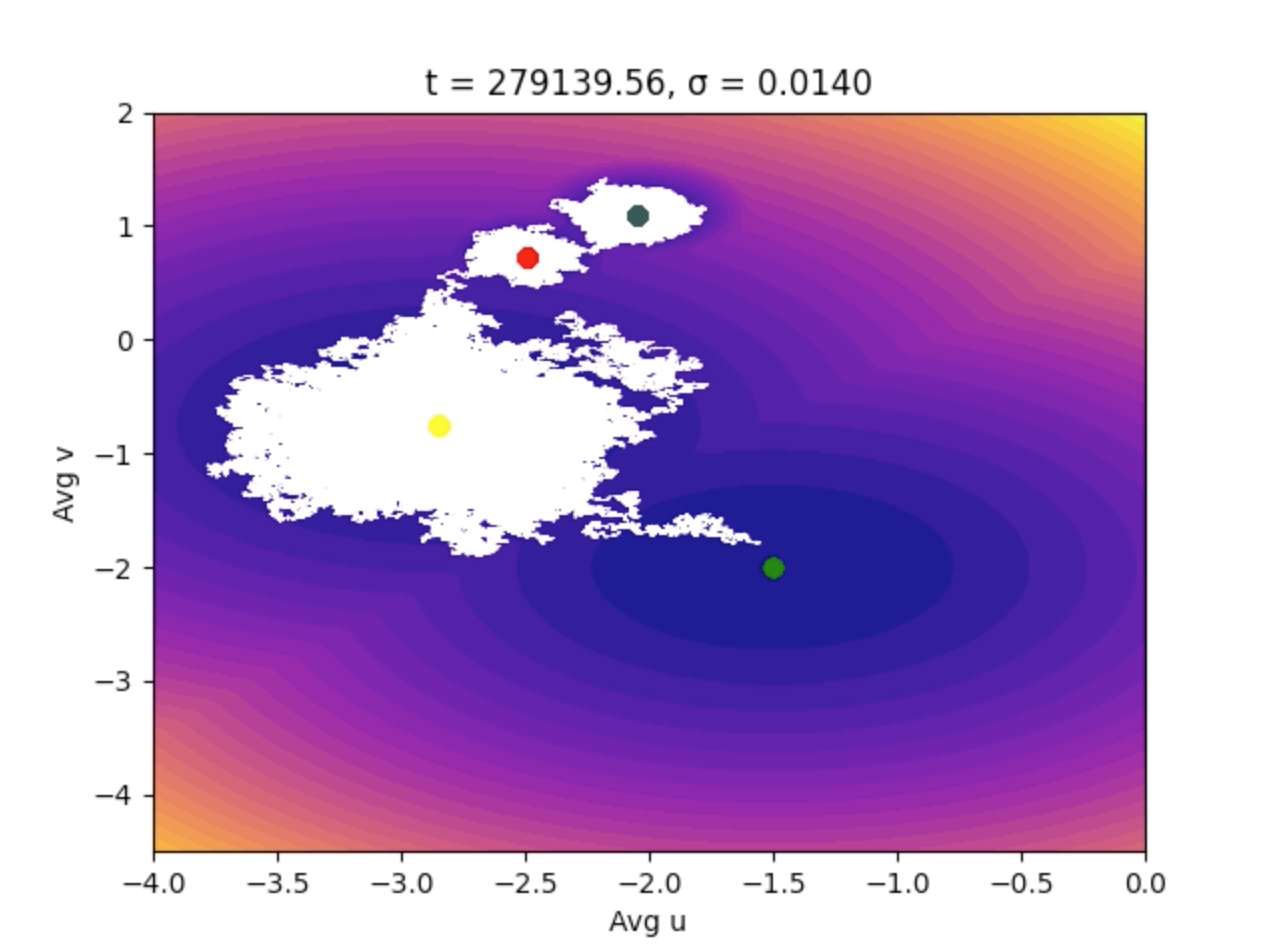} & \includegraphics[scale=0.14]{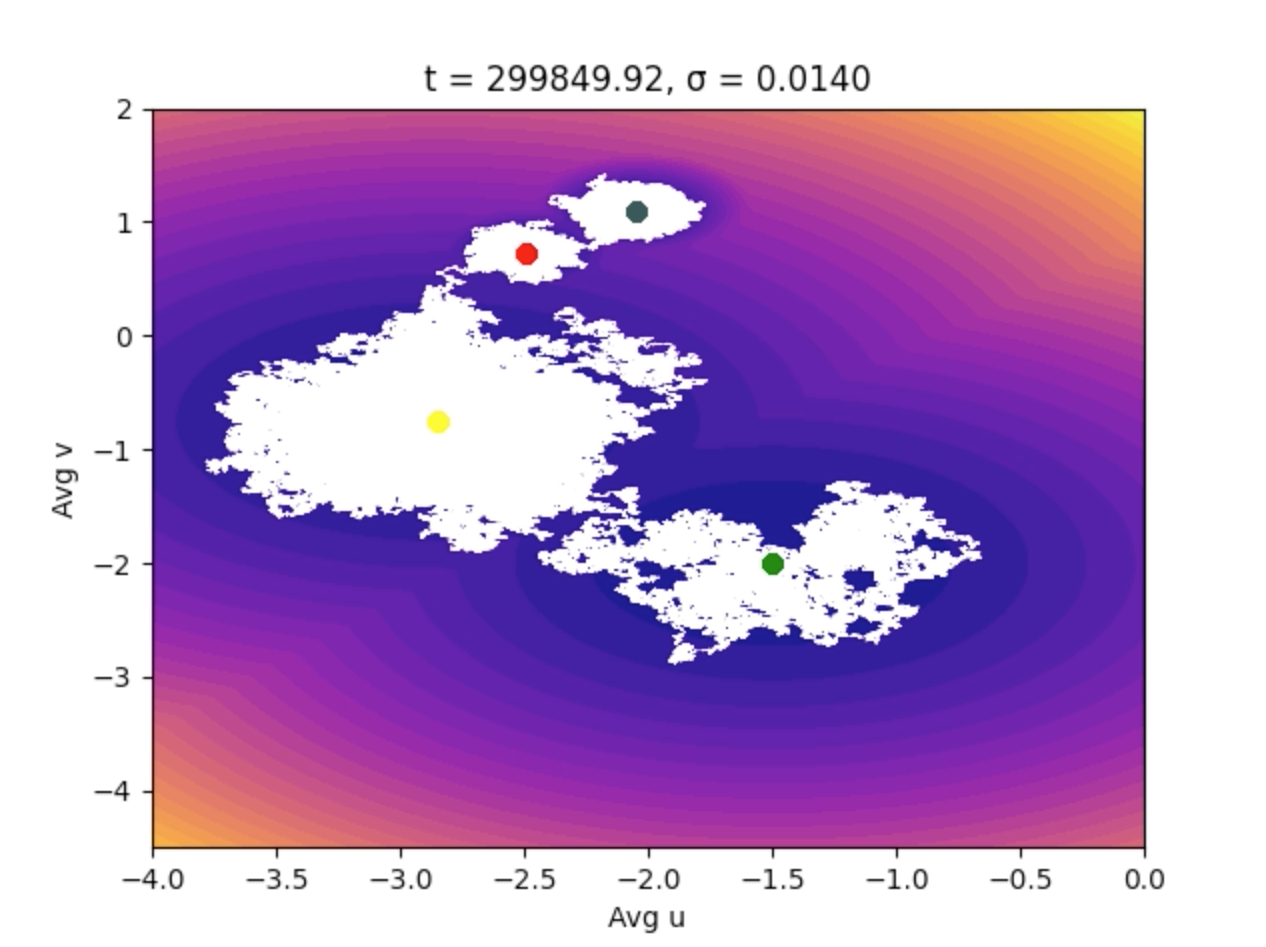}\tabularnewline
\end{tabular}
\par\end{centering}
\caption{Epigenetic landscape and system's evolution (sample path) for different
values of $t$ with $\sigma=0.0140$ and $d_{1}=d_{2}=1$. The white
dot at $t=0.00$ corresponds to the initial condition and Avg is given
by (\ref{eq:avg}).\label{fig:Exit_1}}

\end{figure}

\begin{figure}[H]
\begin{centering}
\begin{tabular}{cc}
\includegraphics[scale=0.14]{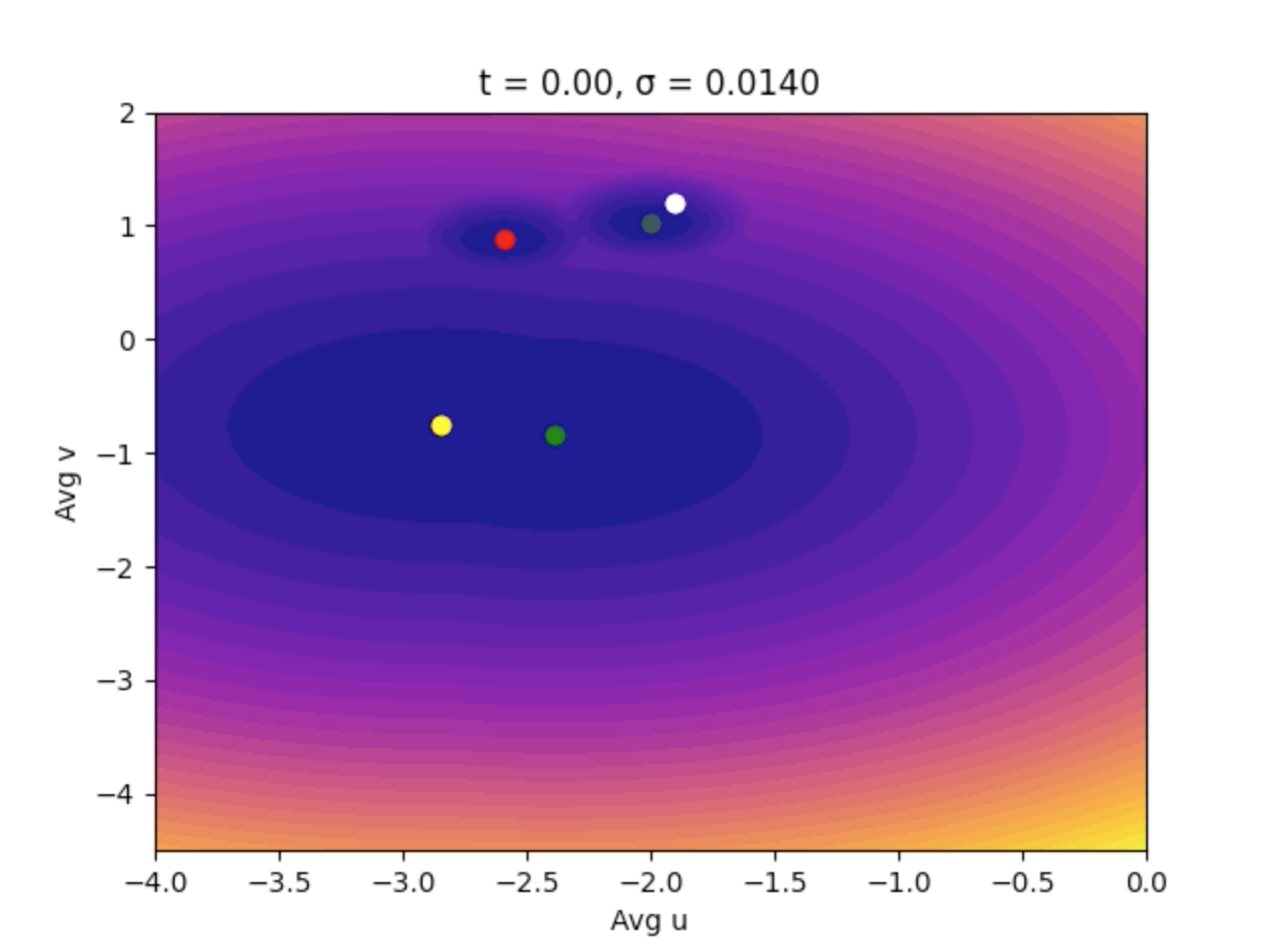} & \includegraphics[scale=0.14]{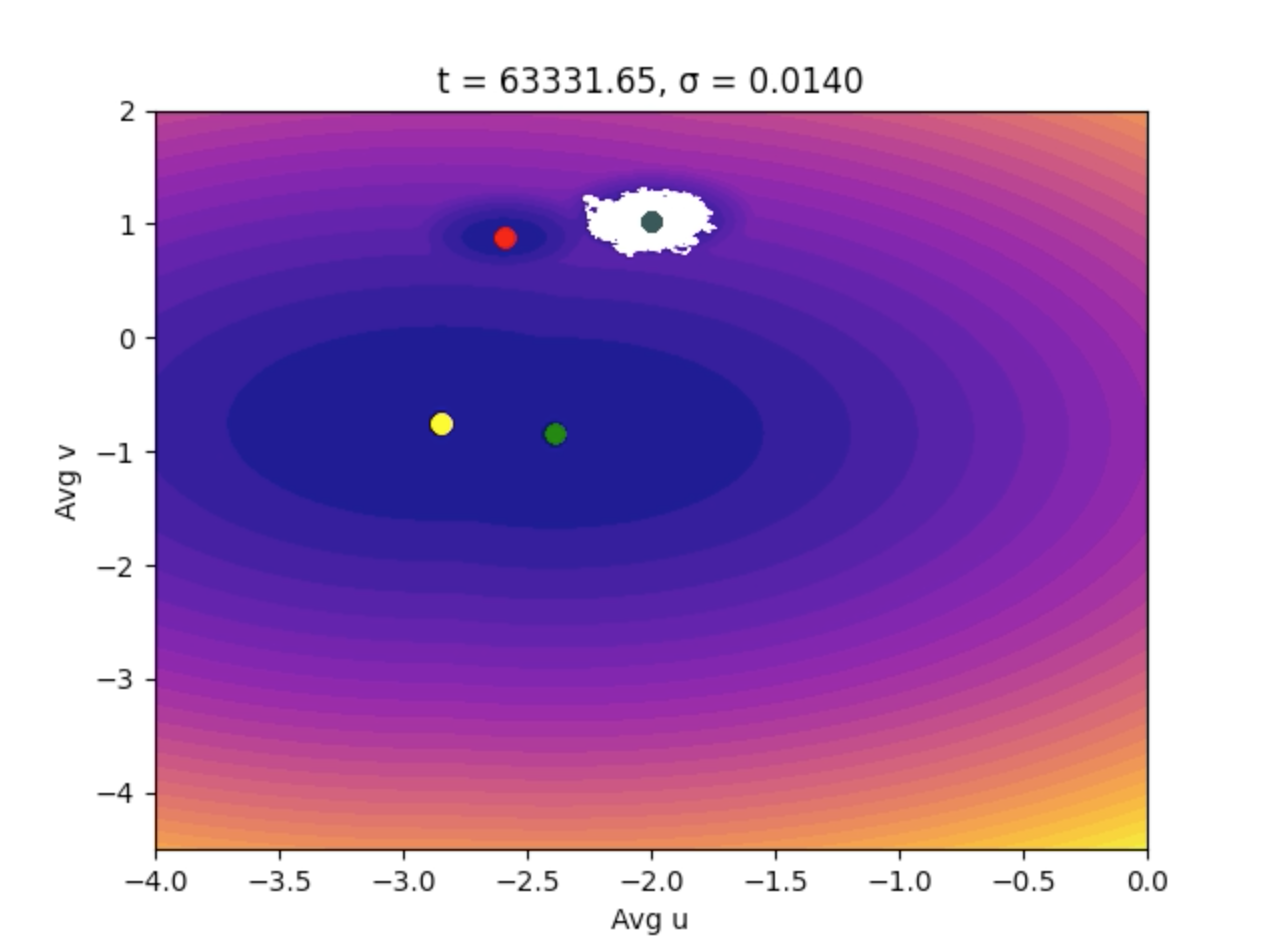}\tabularnewline
\includegraphics[scale=0.14]{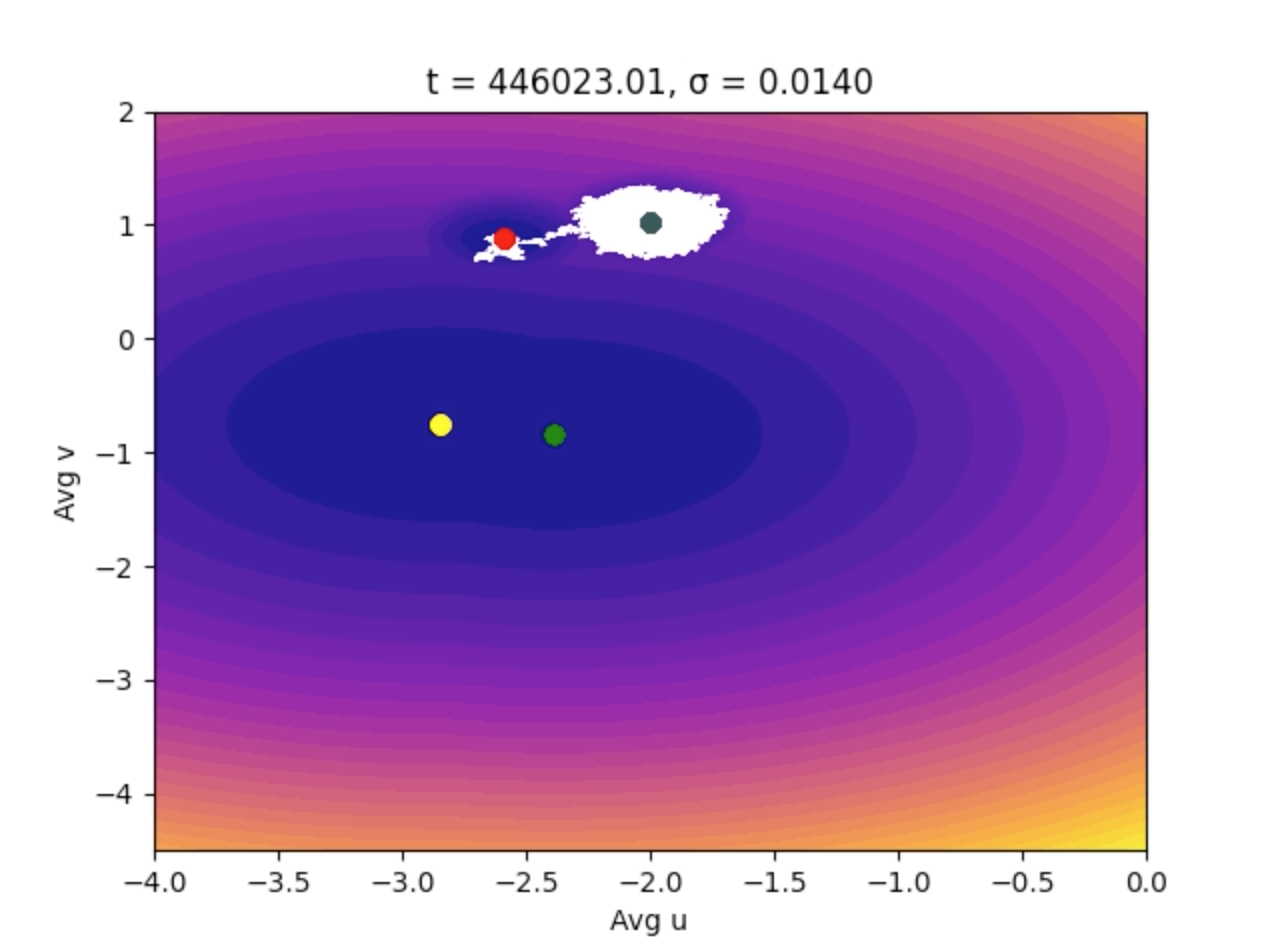} & \includegraphics[scale=0.14]{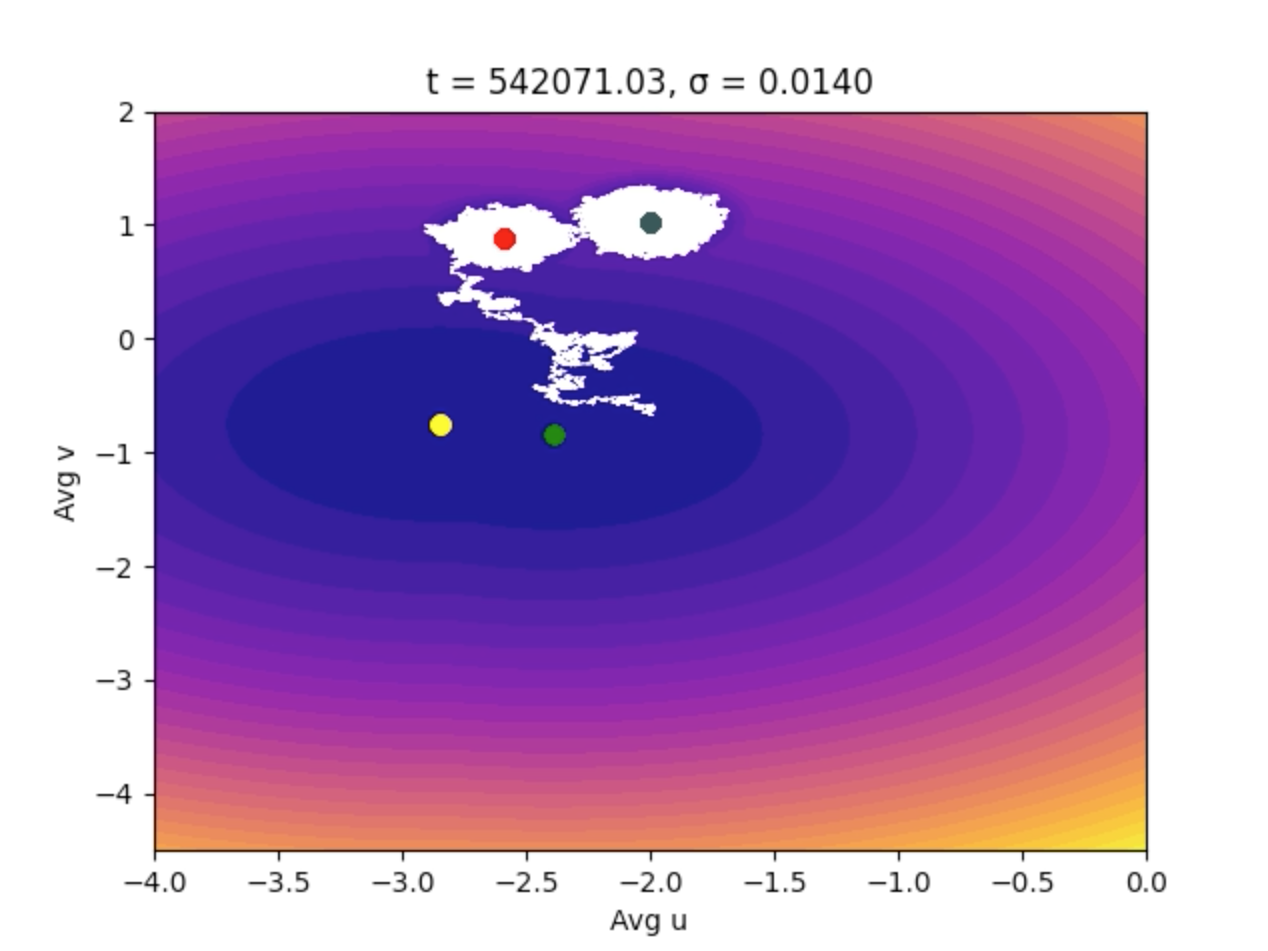}\tabularnewline
\includegraphics[scale=0.14]{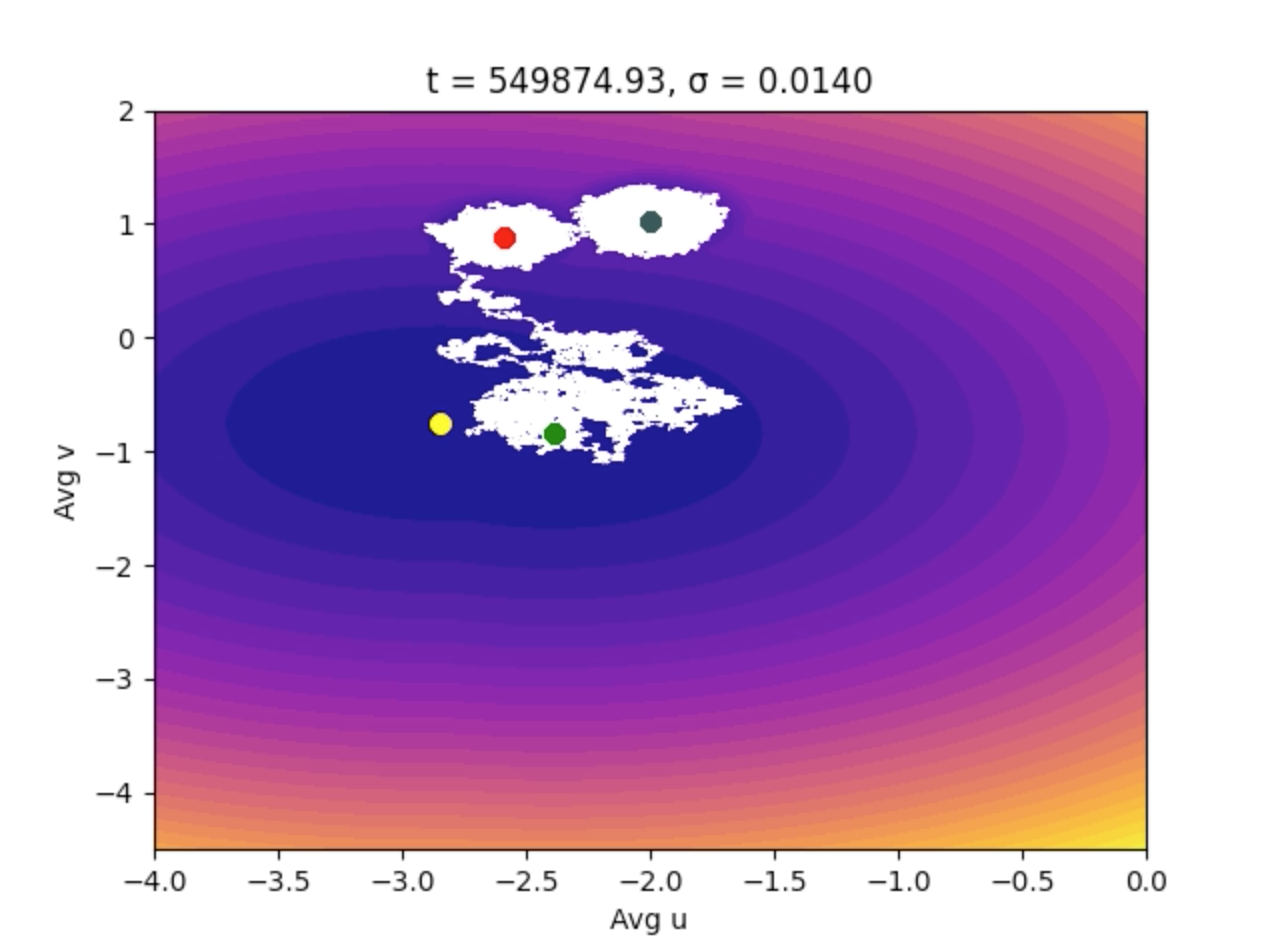} & \includegraphics[scale=0.14]{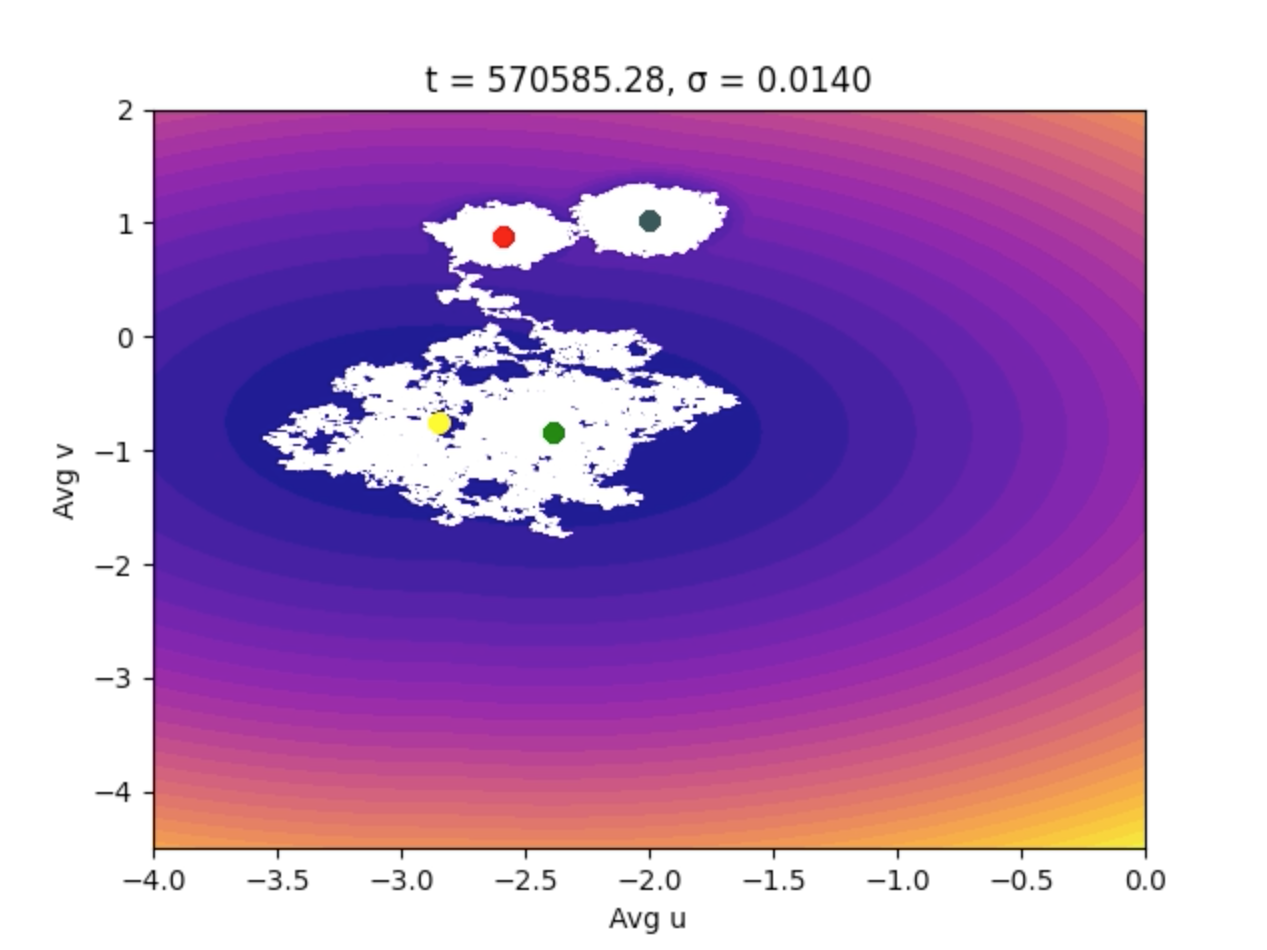}\tabularnewline
\end{tabular}
\par\end{centering}
\caption{Epigenetic landscape (different configuration; cf. Fig (\ref{fig:Exit_1}))
and system's evolution (sample path) for different values of $t$
with $\sigma=0.0140$ and $d_{1}=d_{2}=1$. The white dot at $t=0.00$
corresponds to the initial condition and Avg is given by (\ref{eq:avg})\label{fig:Exit_2}}
\end{figure}

\begin{figure}[H]
\begin{centering}
\begin{tabular}{c}
\includegraphics[scale=0.165]{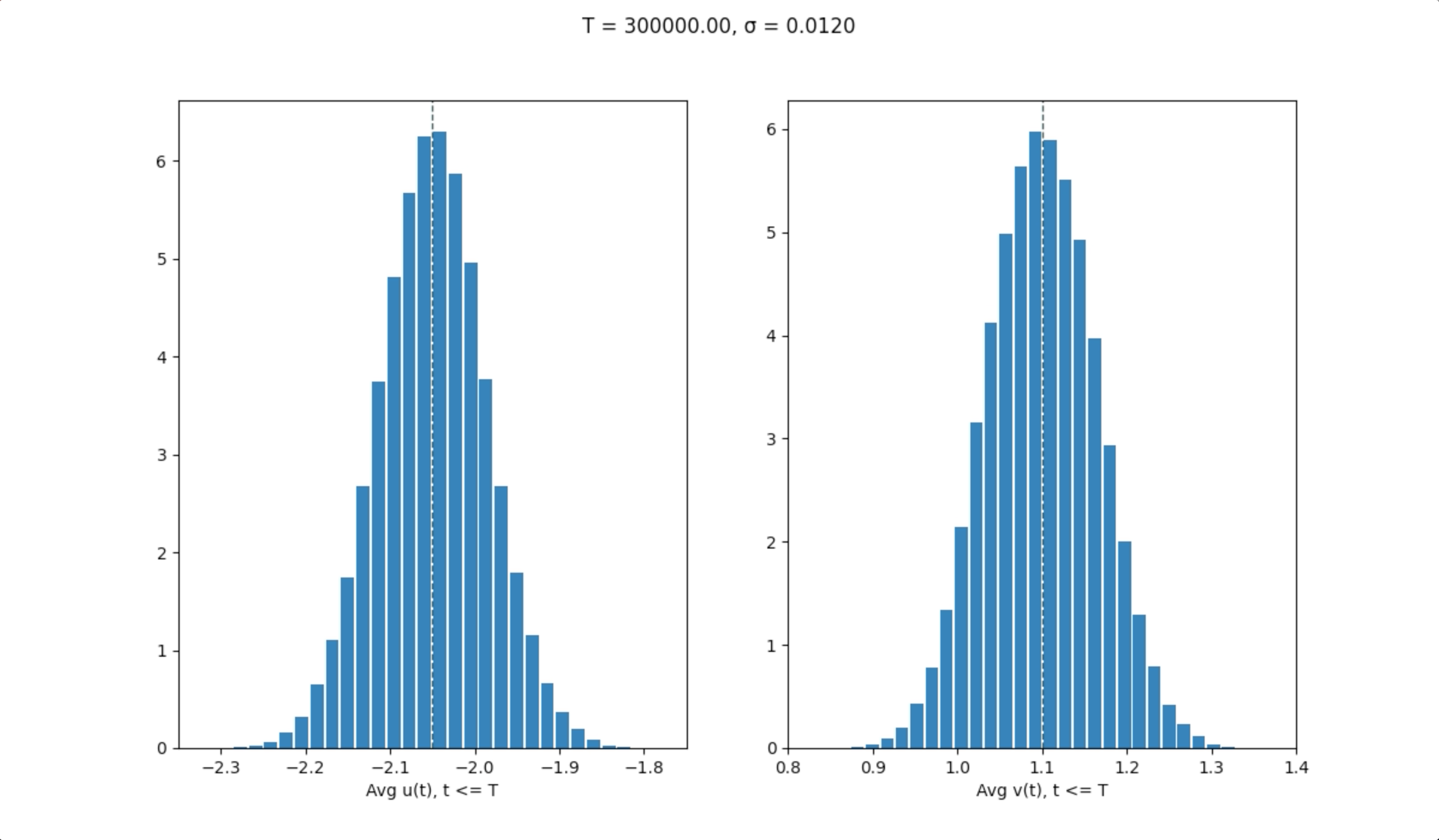}\tabularnewline
\includegraphics[scale=0.165]{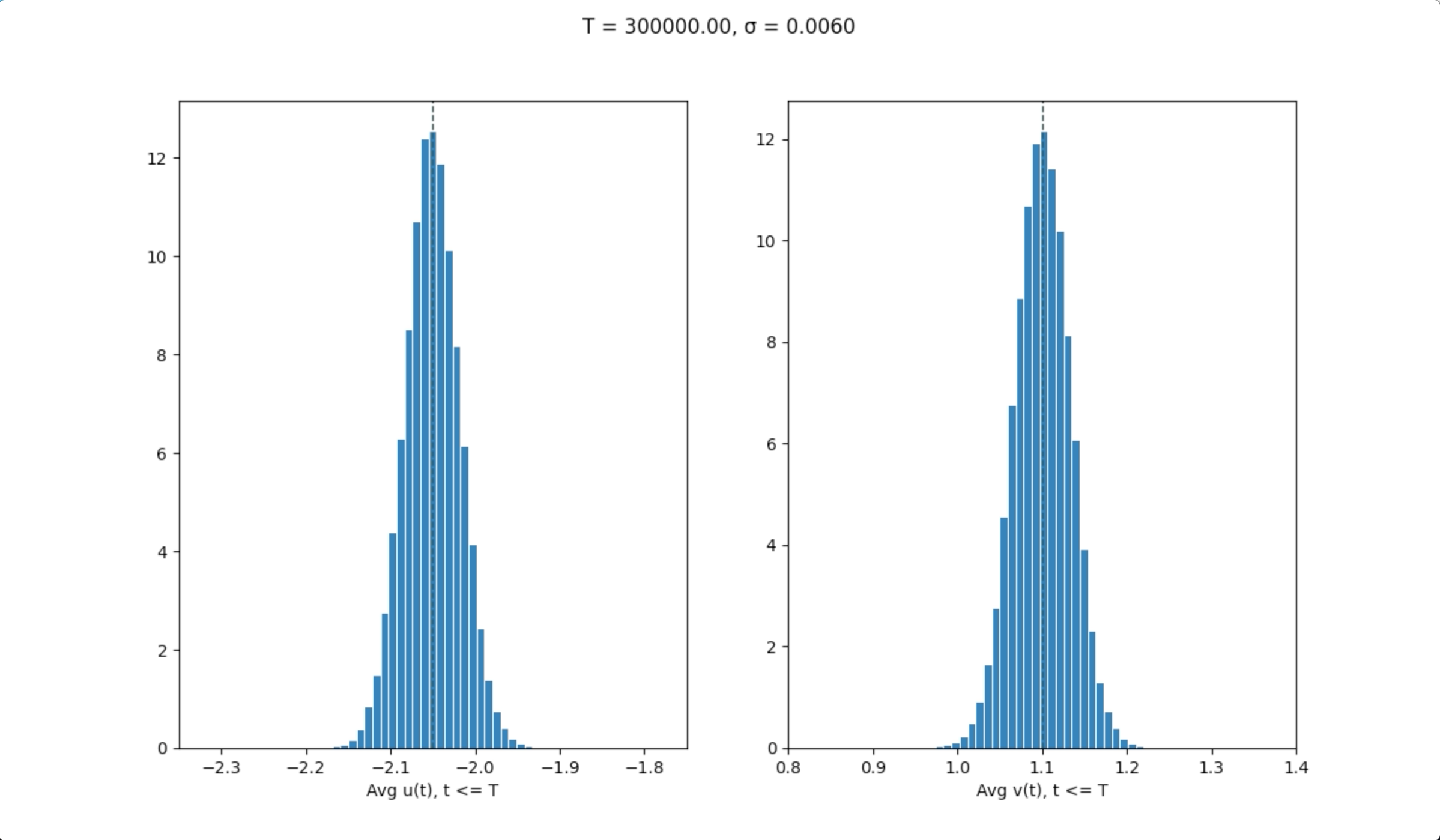}\tabularnewline
\end{tabular}
\par\end{centering}
\caption{\label{fig:Dist1}Distribution of Avg $u\left(t\right)$ and Avg $v\left(t\right)$
(see (\ref{eq:avg})) of a sample path up to time $t=T$ with the
same initial condition as in Fig (\ref{fig:Exit_1}) and for $\sigma=0.0120$
(top), $\sigma=0.0060$ (bottom).}

\end{figure}

\begin{figure}[H]
\begin{centering}
\begin{tabular}{c}
\includegraphics[scale=0.165]{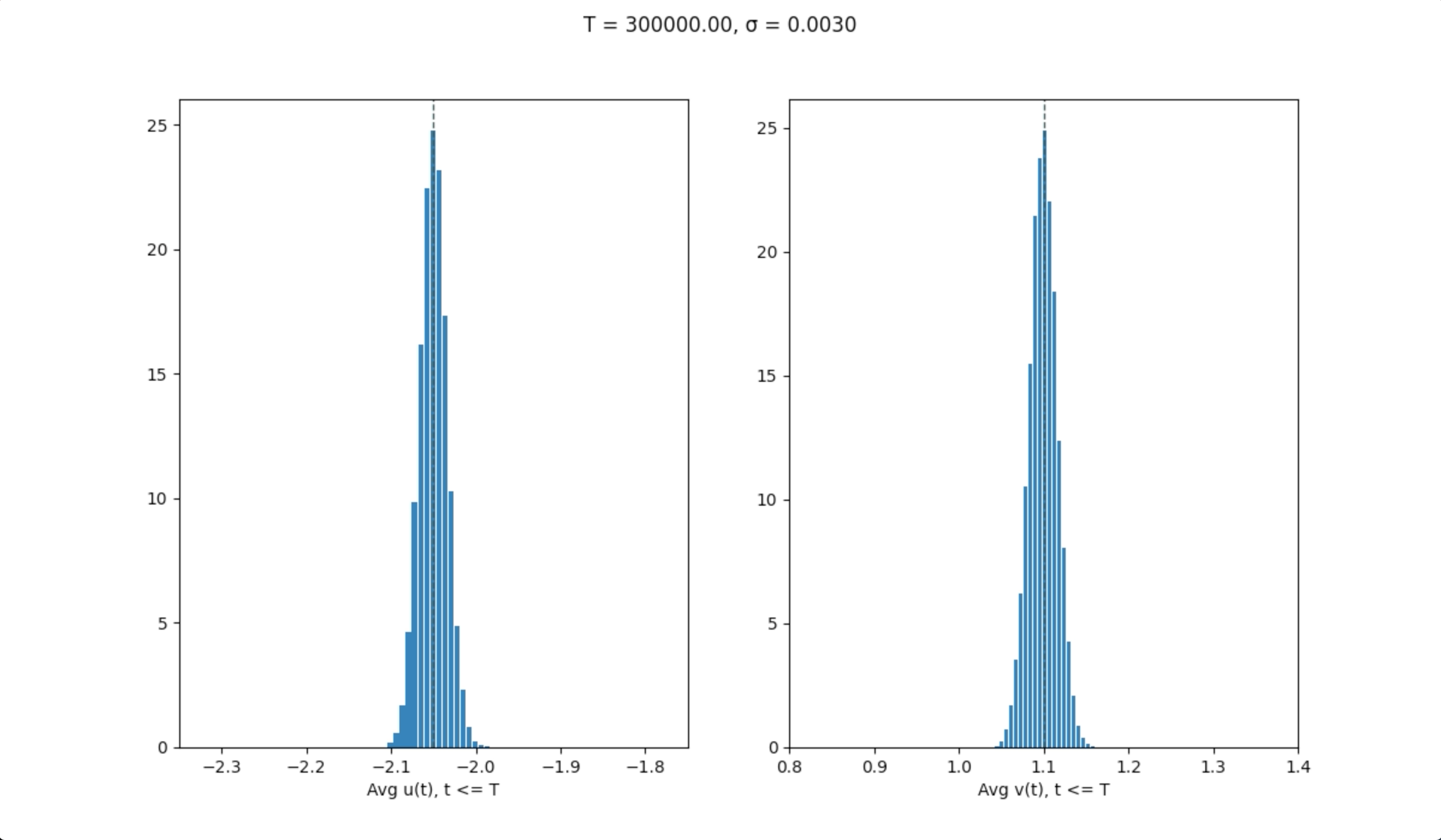}\tabularnewline
\includegraphics[scale=0.165]{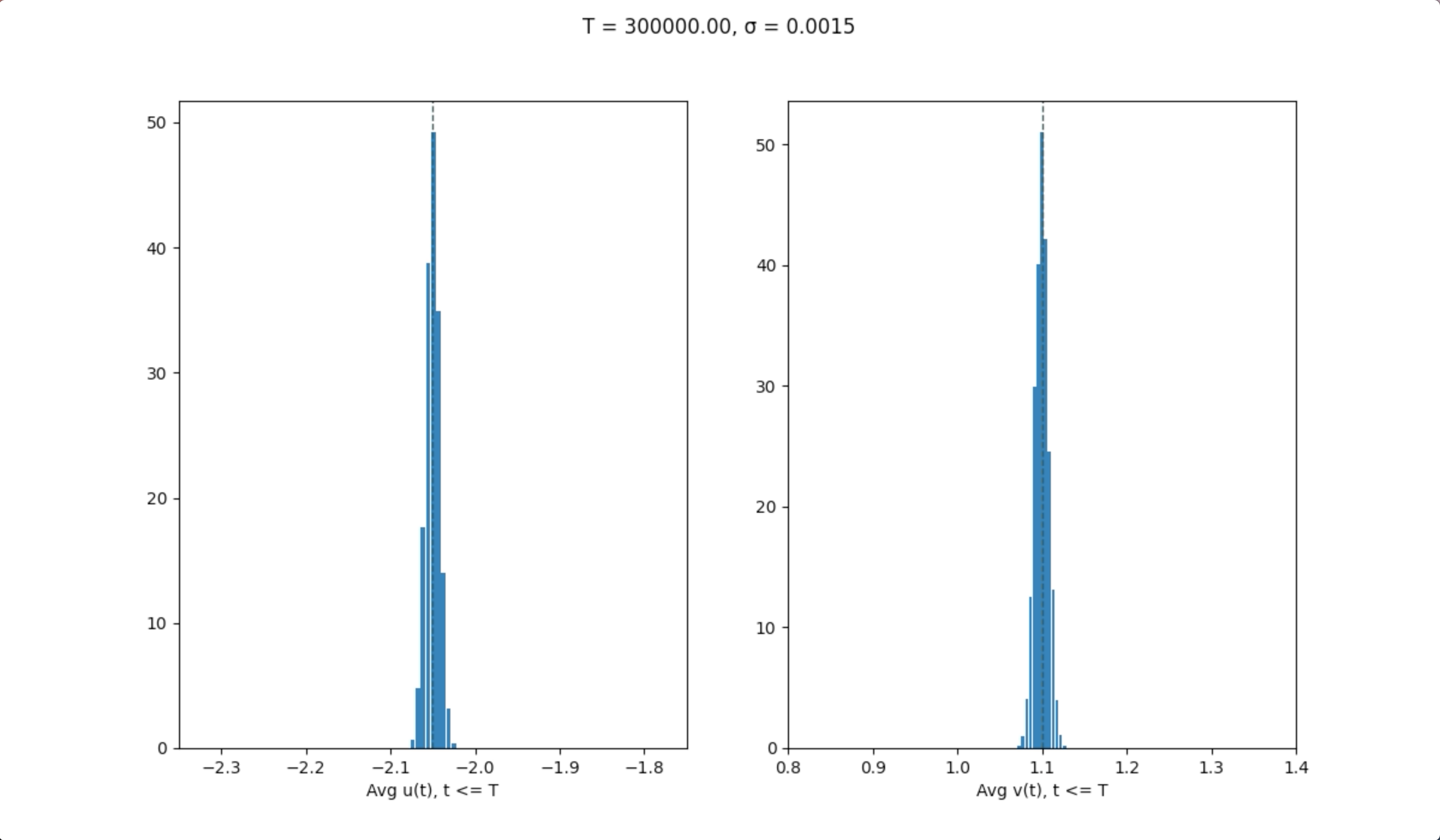}\tabularnewline
\end{tabular}
\par\end{centering}
\caption{\label{fig:Dist2}Distribution of Avg $u\left(t\right)$ and Avg $v\left(t\right)$
(see (\ref{eq:avg})) of a sample path up to time $t=T$ with the
same initial condition as in Fig (\ref{fig:Exit_1}) and for $\sigma=0.0030$
(top), $\sigma=0.0015$ (bottom).}
\end{figure}

\section{Conclusions and Future Work}
We have presented a stochastic model for epigenetic dynamics. Our main goal was to assemble the rigorous theory needed to support our observations, which show how drastically noise and random perturbations may alter the system's behavior. We believe our stochastic model should be coupled with a stochastic differential mean-field game system. In this case, characterizing the long-time behavior, together with the (Malliavin/classical) regularity of solutions, is essential, particularly for numerics, as we do in \cite{sanchez2026regularity}. For many biological systems, it may also be more convenient to consider colored noise, particularly because of correlations. This can make numerical computations easier to handle, and some mathematical results may be easier to obtain, especially in higher dimensions for SPDE. However, colored noise also introduces additional technicalities when applying Large Deviation Theory (LDT). For escape time estimates, more accurate numerical methods should be explored, along with Eyring-Kramers law.

\section{Statements and Declarations}
\subsection{Competing Interests} The authors declare that no competing interests exist.
\subsection{Funding Declaration} No financial support was received for this work.

\begin{appendices}

\renewcommand{\theassumption}{A\arabic{assumption}}
\renewcommand{\theHassumption}{APP.\arabic{assumption}}
\setcounter{assumption}{0} 

\section{Technical Assumptions for the Stochastic Control Framework}\label{sec:appendixA}

\begin{assumption}
\label{assu:S1}Let $a\left(\cdot,\cdot,\cdot\right):\left[0,T\right]\times\Omega\times H\times U\rightarrow H$
and $b\left(\cdot,\cdot,\cdot\right):\left[0,T\right]\times\Omega\times H\times U\rightarrow\mathcal{L}_{2}^{0}$
satisfy:
\begin{enumerate}
\item For any $\left(x,\alpha\right)\in H\times U$, the functions $a\left(\cdot,x,\alpha\right):\left[0,T\right]\times\Omega\rightarrow H$
and $b\left(\cdot,x,\alpha\right)\rightarrow\mathcal{L}_{2}^{0}$
are $\mathbb{F}-$measurable,
\item For any $x\in H$ and a.e. $\left(t,\omega\right)\in\left(0,T\right)\times\Omega$,
the functions $a\left(t,x,\cdot\right):U\rightarrow H$ and $b\left(t,x,\cdot\right):U\rightarrow\mathcal{L}_{2}^{0}$
are continuous,
\item For any $\left(x_{1},x_{2},\alpha\right)\in H\times H\times U$ and
a.e. $\left(t,\omega\right)\in\left(0,T\right)\times\Omega$,
\[
\left\{ \begin{array}{l}
\left|a\left(t,x_{1},\alpha\right)-a\left(t,x_{2},\alpha\right)\right|_{H}+\left|b\left(t,x_{1},\alpha\right)-b\left(t,x_{2},\alpha\right)\right|_{\mathcal{L}_{2}^{0}}\leq C\left|x_{1}-x_{2}\right|_{H},\\
\left|a\left(t,0,\alpha\right)\right|_{H}+\left|b\left(t,0,\alpha\right)\right|_{\mathcal{L}_{2}^{0}}\leq C.
\end{array}\right.
\]
\end{enumerate}
\end{assumption}

\begin{assumption}
\label{assu:s5}For a.e. $\left(t,\omega\right)\in\left(0,T\right)\times\Omega$,
the functions $a\left(t,\cdot,\cdot\right):H\times U\rightarrow H$,
$b\left(t,\cdot,\cdot\right):H\times U\rightarrow\mathcal{L}_{2}^{0}$,
$g\left(t,\cdot,\cdot\right):H\times U\rightarrow\mathbb{R}$, and
$h\left(\cdot\right):H\rightarrow\mathbb{R}$ are $C^{1}$. Furthermore,
for any $\left(x,\alpha\right)\in H\times U$ and a.e. $\left(t,\omega\right)\in\left(0,T\right)\times\Omega$,
we have
\[
\left\{ \begin{array}{l}
\left|a_{x}\left(t,x,\alpha\right)\right|_{\mathcal{L}\left(H\right)}+\left|b_{x}\left(t,x,\alpha\right)\right|_{\mathcal{L}\left(H;\mathcal{L}_{2}^{0}\right)}+\left|g_{x}\left(t,x,\alpha\right)\right|_{H}+\left|h_{x}\left(x\right)\right|_{H}\leq C,\\
\left|a_{\alpha}\left(t,x,\alpha\right)\right|_{\mathcal{L}\left(\tilde{H};H\right)}+\left|b_{\alpha}\left(t,x,\alpha\right)\right|_{\mathcal{L}\left(\tilde{H};\mathcal{L}_{2}^{0}\right)}+\left|g_{\alpha}\left(t,x,\alpha\right)\right|_{\mathcal{L}\left(\tilde{H}\right)}\leq C
\end{array}\right.
\]
\end{assumption}

\begin{assumption}
\label{assu:As1}$a\left(\cdot,\cdot,\cdot,\cdot\right):\left[0,T\right]\times H\times H_{1}\times\Omega\rightarrow H$
and $b\left(\cdot,\cdot,\cdot,\cdot\right):\left[0,T\right]\times H\times H_{1}\times\Omega\rightarrow\mathcal{L}_{2}^{0}$
are two maps such that
\begin{enumerate}
\item For any $\left(x,\alpha\right)\in H\times H_{1}$, $a\left(\cdot,x,\alpha,\cdot\right):\left[0,T\right]\times\Omega\rightarrow H$
and $b\left(\cdot,x,\alpha,\cdot\right):\left[0,T\right]\times\Omega\rightarrow\mathcal{L}_{2}^{0}$
are $\mathcal{B}\left(\left[0,T\right]\right)\times\mathcal{F}$ measurable
and $\mathbb{F}-$adapted,
\item For any $\left(t,x,\omega\right)\in\left[0,T\right]\times H\times\Omega$,
$a\left(t,x,\cdot,\omega\right):H_{1}\rightarrow H$ and $b\left(t,x,\cdot,\omega\right):H_{1}\rightarrow\mathcal{L}_{2}^{0}$
are continuous and
\[
\begin{cases}
\left|a\left(t,x_{1},\alpha,\omega\right)-a\left(t,x_{2},\alpha,\omega\right)\right|_{H}+\left|b\left(t,x_{1},\alpha,\omega\right)-b\left(t,x_{2},\alpha,\omega\right)\right|_{\mathcal{L}_{2}^{0}}\leq\\
\hfill C\left|x_{1}-x_{2}\right|_{H}\quad\forall\left(t,x_{1},x_{2},\alpha,\omega\right)\in\left[0,T\right]\times H\times H\times H_{1}\times\Omega\\
\left|a\left(t,0,\alpha,\omega\right)\right|_{H}+\left|b\left(t,0,\alpha,\omega\right)\right|_{\mathcal{L}_{2}^{0}}\leq C,\quad\forall\left(t,\alpha,\omega\right)\in\left[0,T\right]\times H_{1}\times\Omega.
\end{cases}
\]
\end{enumerate}
\end{assumption}

\begin{assumption}
\label{assu:As2}For a.e. $\left(t,\omega\right)\in\left[0,T\right]\times\Omega$,
the functions $a\left(t,\cdot,\cdot,\omega\right):H\times H_{1}\rightarrow H$
and $b\left(t,\cdot,\cdot,\omega\right):H\times H_{1}\rightarrow\mathcal{L}_{2}^{0}$
are differentiable, and $\left(a_{x}\left(t,x,\alpha,\omega\right),a_{\alpha}\left(t,x,\alpha,\omega\right)\right)$
and $\left(b_{x}\left(t,x,\alpha,\omega\right),a_{\alpha}\left(t,x,\alpha,\omega\right)\right)$
are uniformly continuous with respect to $x\in H$ and $\alpha\in U$.
There exists a nonnegative $\eta\in L_{\mathbb{F}}^{2}\left(0,T;\mathbb{R}\right)$
such that for a.e. $\left(t,\omega\right)\in\left[0,T\right]\times\Omega$
and for all $x\in H$ and $\alpha\in H_{1}$,
\[
\begin{cases}
\left|a\left(t,0,\alpha,\omega\right)\right|_{H}+\left|b\left(t,0,\alpha,\omega\right)\right|_{\mathcal{L}_{2}^{0}}\leq C\left(\eta\left(t,\omega\right)+\left|\alpha\right|_{H_{1}}\right),\\
\left|a_{x}\left(t,x,\alpha,\omega\right)\right|_{\mathcal{L}\left(H\right)}+\left|a_{\alpha}\left(t,x,\alpha,\omega\right)\right|_{\mathcal{L}\left(H_{1};H\right)}+\left|b_{x}\left(t,x,\alpha,\omega\right)\right|_{\mathcal{L}\left(H;\mathcal{L}_{2}^{0}\right)}\\
\hfill+\left|b_{\alpha}\left(t,x,\alpha,\omega\right)\right|_{\mathcal{L}\left(H_{1};\mathcal{L}_{2}^{0}\right)}\leq C.
\end{cases}
\]
\end{assumption}

\end{appendices}


\end{document}